\documentclass[11pt]{article}

\usepackage[letterpaper,margin=0.88in]{geometry}
\usepackage{ifpdf}
\usepackage{amsmath,amsthm,amsfonts,amssymb}

\usepackage[T1]{fontenc}
\usepackage{lmodern}
\usepackage{microtype}

\usepackage{amsmath,amssymb,amsthm,mathtools}
\usepackage{graphicx}
\usepackage{booktabs}
\usepackage{enumitem}
\usepackage[font=small,labelfont=bf]{caption}
\usepackage[hidelinks]{hyperref}
\usepackage[capitalise,nameinlink]{cleveref}

\usepackage{mathtools}

\makeatletter
\renewcommand{\maketitle}{%
	\begin{center}
		{\LARGE \@title \par}
		\vskip 1.5em
		{\large
			\lineskip .5em
			\begin{tabular}[t]{c}
				\@author
			\end{tabular}\par}
	\end{center}
	\par\vskip 1.5em
}
\makeatother

\usepackage{lipsum}
\usepackage{amsfonts}
\usepackage{graphicx}
\usepackage{xcolor}
\colorlet{RED}{red}
\colorlet{GREEN}{green!50!black}
\colorlet{BLACK}{black}
\usepackage{epstopdf}
\usepackage{algorithmic}
\usepackage{fancyhdr,subfigure}
\ifpdf
\DeclareGraphicsExtensions{.eps,.pdf,.png,.jpg}
\else
\DeclareGraphicsExtensions{.eps}
\fi

\usepackage{booktabs}
\usepackage{hyperref}
\usepackage{stmaryrd}
\usepackage{bm}
\usepackage{caption}
\usepackage{multirow}
\usepackage{enumitem}
\usepackage{amssymb}
\usepackage{amsopn}
\usepackage{cleveref}

\theoremstyle{plain}
\newtheorem{theorem}{Theorem}[section]
\newtheorem{lemma}[theorem]{Lemma}
\newtheorem{proposition}[theorem]{Proposition}
\newtheorem{corollary}[theorem]{Corollary}

\theoremstyle{definition}
\newtheorem{definition}[theorem]{Definition}

\theoremstyle{remark}
\newtheorem{remark}[theorem]{Remark}

\providecommand{\email}[1]{\texttt{#1}}

\newenvironment{keywords}
{\par\medskip\noindent\textbf{Keywords.}\ }
{\par\medskip}

\newenvironment{AMS}
{\par\medskip\noindent\textbf{Mathematics Subject Classification.}\ }
{\par\medskip}

\title{Hidden Accuracy and Superconvergence Analysis of Central Discontinuous Galerkin Methods on Overlapping Meshes 
	\thanks{
			This work was partially supported by Science Challenge Project (No.~TZ2025007), Shenzhen Science and Technology Program (Grant Nos.~JCYJ20250604144300001 and RCJC20221008092757098), and 
			National Natural Science Foundation of China (No.~124B2022).}}
	
	\author{Manting Peng\thanks{Department of Mathematics, Southern University of Science and Technology, Shenzhen, Guangdong 518055, China  (\email{pengmt2024@mail.sustech.edu.cn}).}
		\and Kailiang Wu\thanks{Corresponding author. Department of Mathematics and Shenzhen International Center for Mathematics, Southern University of Science and Technology, Shenzhen, Guangdong 518055, China (\email{wukl@sustech.edu.cn}).}} 

\ifpdf
\hypersetup{
	pdftitle={},
	pdfauthor={}
}
\fi

\allowdisplaybreaks

\newlist{steps}{enumerate}{1}
\setlist[steps, 1]{label =\textbf{ Step \arabic*:}}

\usepackage{multirow}

\begin{document}

	\maketitle

\begin{abstract}
This paper establishes the first rigorous superconvergence theory for semidiscrete and fully discrete central discontinuous Galerkin (CDG) methods for linear hyperbolic equations on overlapping meshes. While the optimal $L^2$ convergence of $\mathbb{Q}^k$ CDG schemes was established on uniform Cartesian meshes by Liu, Shu, and Zhang~[{\em SIAM J.~Numer.~Anal.}, 56 (2018), pp.~520--541], {\it their observed $\mathcal{O}(h^{k+2})$ pointwise superconvergence has remained unproven}, due to the loss of standard single-mesh Galerkin orthogonality inherent in the CDG overlapping structure. 
To overcome this fundamental barrier, we introduce a projection-correction framework that identifies a hidden superconvergent mechanism: an asymptotic weak residual cancellation in one dimension, and a high-order cancellation-by-aggregation (HOCA) mechanism in multiple dimensions. This HOCA approach overcomes the analytical challenge posed by coupled primal-dual directional residuals, recovering critical error cancellation properties absent from the standard variational formulation. Consequently, we provide the rigorous proof of the conjectured $\mathcal{O}(h^{k+2})$ pointwise superconvergence in the discrete $\ell^{\infty}$ norm across all superconvergent points. Furthermore, we reveal that under a systematically corrected initialization, this framework yields a previously undiscovered, stronger cell-average superconvergence estimate of order $\mathcal{O}(h^{\min\{2k+1,k+3\}})$. The theory is extended to fully discrete explicit Runge--Kutta CDG schemes, where stagewise corrected errors are constructed to preserve spatial superconvergence up to temporal truncation errors, yielding a stable reconstruction-based postprocessing estimate.  Numerical experiments in one and two spatial dimensions confirm the sharpness of the theoretical rates.
\end{abstract}

\begin{keywords}
	Central discontinuous Galerkin methods, hyperbolic equations, superconvergence, projection-correction, high-order cancellation-by-aggregation, Runge--Kutta methods
\end{keywords}

	\begin{AMS}
		65M60, 65M12, 65M15, 35L65
	\end{AMS}

\section{Introduction}
Among the many high-order methods for hyperbolic problems, discontinuous Galerkin (DG) methods \cite{cockburn1989tvb,cockburn1998runge,cockburn2001runge,peng2025oedg} are widely used due to their local conservation, high-order accuracy, and natural compatibility with both explicit time discretizations and complex geometries. Within this family, the central discontinuous Galerkin (CDG) method introduced by Liu \emph{et al.}~\cite{liu2007central} uses two overlapping meshes, referred to as the primal and dual meshes, and defines discontinuous polynomial spaces on each. This overlapping structure eliminates the need for numerical fluxes and Riemann solvers, providing a ``black-box'' spatial discretization for hyperbolic equations. CDG methods have been successfully applied to a variety of physical systems \cite{li2011central,wu2017physical,wu2023provably,ding2024gql}. While the optimal-order error estimates for semidiscrete $\mathbb{Q}^k$ CDG schemes on uniform Cartesian meshes were established in \cite{liu2018optimal,jiao2022optimal}, those analyses rely on specially designed projections adapted to the spatial operator. More precisely, the projection error satisfies a high-order residual estimate in the spatial operator, which is then used to obtain the optimal \(L^2\) convergence rate of order \(k+1\). 
Consequently, this operator-level estimate is distinct from the pointwise and cell-average superconvergence explored in this work.

Beyond the optimal \(L^2\) theory, numerical evidence points to stronger convergence phenomena: point values at specific points and cell averages converge faster than global error. This phenomenon is important both analytically and computationally, since it identifies quantities that can be used for smoothness indication or postprocessing. For standard DG methods, the corresponding superconvergence theory is by now extensive; see, e.g., \cite{adjerid_flaherty_krivodonova_2002,adjerid_massey_2006,biswas_devine_flaherty_1994,cockburn_luskin_shu_suli_2003,xie_zhang_2010,zhang_xie_zhang_2009} for steady and singularly perturbed problems, and \cite{cheng_shu_2008,cheng_shu_2010,meng_shu_zhang_wu_2012,yang_shu_2012,cao_zhang_zou_2014,cao_li_yang_zhang_2017,cao_liu_zhang_2017,cao_shu_yang_zhang_2015,xu_meng_shu_zhang_2020,xu_zhang_2022,lu2021oscillation,chen2020superconvergence,cockburn2017superconvergence} for time-dependent hyperbolic and convection--diffusion equations. For the one-dimensional (1D) linear advection equation with an upwind numerical flux, for example, a degree-$k$ DG solution exhibits \((k+2)\)nd-order superconvergence at Radau points and $(2k+1)$st-order superconvergence in cell averages and at certain downwind points. A key tool in this theory is the correction-function framework of Cao, Zhang, and Zou \cite{cao_zhang_zou_2014}, which refines Gauss--Radau projection arguments and explains several DG superconvergent mechanisms in a unified way.

In 2018, Liu, Shu, and Zhang \cite{liu2018optimal} established the optimal $L^2$ convergence of semidiscrete $\mathbb{Q}^k$ CDG schemes  on uniform Cartesian meshes, 
and they also observed a $\mathcal{O}(h^{k+2})$ pointwise superconvergence for CDG schemes. 
However, the rigorous proof of their conjectured superconvergence phenomenon had remained open due to the loss of standard single-mesh Galerkin orthogonality inherent in the CDG overlapping structure. 
This raises a natural question: 
 {\em which feature of the overlapping-mesh structure accounts for the pointwise accuracy observed in computations but not captured by the global $L^2$ estimate?} However, for CDG methods, even the benchmark case of the linear advection equation on uniform Cartesian overlapping meshes has lacked a comparable mechanism-level explanation.  The obstruction to superconvergent analysis for CDG methods is structural: the staggered primal and dual meshes of CDG methods lead to coupled primal-dual residuals in the error equation rather than a single-mesh DG residual, and the local Galerkin orthogonality used in classical correction-function analysis for standard DG methods is no longer available. {\em Thus the present theory is not obtained by simply combining the Liu--Shu--Zhang (LSZ) projector with the Cao-type DG correction hierarchy.} In fact, we show that the observed pointwise and cell-average superconvergence is governed by the order of the primal-dual coupling term in the residual. By decomposing the Galerkin test space into cell-average (piecewise-constant functions) and zero-cell-average components, we identify the former as the key component in the superconvergence analysis, which we term {\it principal residual cancellation}. In 1D, we prove that the principal residual cancellation satisfies a higher-order estimate than that suggested by the standard $L^2$-error analysis. The remaining zero-cell-average component is then controlled via specially constructed correction functions with zero cell averages. Together, these two estimates provide the error bound necessary to establish the pointwise and cell-average superconvergence. In two-dimensional (2D) and higher-dimensional problems, the same cancellation fails term-by-term because the tensor-product projector has no single variational characterization and the coordinate directions are coupled through the overlapping meshes. The remedy is directional correction together with a high-order cancellation-by-aggregation (HOCA) principle: the directional projection residual alone does not yield the sharp asymptotic order, while the aggregate of the directional projection residual and all same-direction correction residuals does. Thus, the contribution is not a direct application of classical DG superconvergence theory to CDG methods, but a mechanism-level superconvergence analysis.

This paper establishes a rigorous superconvergence theory for semidiscrete and fully
discrete \(\mathbb Q^k\)-based CDG methods for the linear advection equation on uniform Cartesian overlapping meshes.
The first part of the theory develops a corrected-error residual-cancellation analysis. The core mechanism relies on the corrections to eliminate the dominant projector-induced primal-dual error, thereby exposing the key residual cancellation for superconvergence. Conditioned on a corrected initialization, this framework establishes a previously undiscovered, stronger $\mathcal{O}(h^{\min\{2k+1,k+3\}})$ cell-average estimate. The second part proves the observed pointwise phenomenon under the LSZ projection initialization in the \(\ell^\infty\)-norm over all superconvergent points. Its proof relies on the discrete $\ell_h^2$ estimate at superconvergent points and their difference quotients with an anchored discrete Sobolev inequality, and therefore explains why the pointwise theorem avoids the \(h^{-d/2}\) loss that would follow from a direct inverse inequality applied to a global \(L^2\) error bound, where $d$ is the spatial dimension.

The main contributions are as follows.
\begin{itemize}
	\item We discover the hidden residual-cancellation mechanism behind superconvergence: asymptotic weak constant-test cancellation in 1D and directional aggregation through HOCA in multiple dimensions. This mechanism is distinct from both the optimal-error LSZ projector and the strong single-mesh Galerkin orthogonality used in the Cao-type DG correction theory.
	\item For semidiscrete CDG schemes, we identify and rigorously prove a stronger $\mathcal{O}(h^{\min\{2k+1,k+3\}})$ cell-average superconvergent estimate under a corrected initialization.  
	\item We prove the pointwise superconvergence rate under the LSZ projection initialization in the $\ell^{\infty}$-norm over all superconvergent points. The proof combines $\ell_h^2$ pointwise estimates for superconvergent points and their difference quotients, and a discrete Sobolev inequality at chosen points.  
	\item For multidimensional overlapping Cartesian meshes, 
	we construct directional correction functions based on Gauss--Lobatto projections and LSZ projections, and prove the HOCA residual estimate needed to close the tensor-product analysis. 
	We emphasize that the multidimensional weak cancellation is not inherited by each directional residual separately, because the tensor-product projection has no single local variational characterization and the primal--dual overlap couples transverse and mixed terms. The HOCA argument shows that the sharp order is recovered only after the projection and same-direction correction residuals are aggregated. This feature explains why the multidimensional theory is nontrivial and not a routine tensorization of the 1D analysis.
	\item For explicit RKCDG schemes, we construct stagewise corrected errors and prove that the fully discrete method preserves the same spatial superconvergent orders, up to temporal truncation errors, under the stability conditions of \cite{peng2025oscillation}.
	\item We prove that accurate cell averages can be converted, through a stable reconstruction-based postprocessor, into a higher-order polynomial approximation.
	\item The numerical experiments confirm the predicted cell-average and postprocessing approximation rates. 
\end{itemize}
It is important to distinguish the two initializations. The uncorrected LSZ
projection initialization is sufficient for the maximum-norm pointwise estimate
at the superconvergent points. The higher-order cell-average estimate, however,
requires a corrected initialization, because the corrected error must be small
initially at order \(h^{\sigma_k}\).

As in many existing superconvergence results for standard, non-central DG methods
\cite{cheng_shu_2010,yang_shu_2012,cao_zhang_zou_2014,cao_shu_yang_zhang_2015,cao2018superconvergence},
the first superconvergence analysis for CDG methods presented here is restricted
to overlapping Cartesian grids. To the best of our knowledge, an analogous
superconvergence theory on general triangular meshes, or more broadly on
non-Cartesian meshes, is not yet available even for many standard DG schemes for
hyperbolic problems. This is largely due to the absence, or the current lack of
understanding, of the relevant superconvergent point structure on such meshes.
Whether CDG methods exhibit superconvergence on non-Cartesian overlapping meshes
therefore remains an open problem.

The paper is organized as follows. Section~2 reviews the CDG scheme, including stability and optimal-error estimates. Sections~3 and~4 develop the projection--correction framework in 1D and multiple dimensions, respectively; the final subsection of Section~4 proves the LSZ-projection-based \(\ell^\infty\)-norm pointwise superconvergence theorem and discusses the extension to higher Cartesian dimensions. Section~5 presents the fully discrete RKCDG estimates and the reconstruction-based postprocessing result. Sections~6 and~7 report numerical tests and conclude the paper.

\textbf{Notation.}
Throughout the paper, $A\lesssim B$ means $A\le MB$ for a generic constant $M>0$ independent of the mesh size $h$
(and of $\Delta t$ in the fully discrete setting). The constant $M$ may change from line to line. For any cell $K$, we write
\[
\|w\|_K^2 := \int_K |w|^2\,dx,\qquad \|w\|^2 := \sum_K \|w\|_K^2,
\]
and use $(\cdot,\cdot)_K$ for the $L^2$ inner product on $K$. We denote by $\mathbb{Q}^k$ the tensor-product space of polynomials of degree at most $k$ in each coordinate direction.

\section{Semidiscrete CDG method}\label{sec:semiCDG}

We consider the linear advection equation with periodic boundary conditions
\begin{equation}\label{equ:conservation_law}
	\partial_t u + \boldsymbol{\beta}\cdot\nabla u = 0,
\end{equation}
where \( u = u(\bm{x}, t) \) is the exact solution, \( u_0(\bm{x}) \) is the initial condition, and $\boldsymbol{\beta} = (\beta_1, \dots, \beta_d)$ is a constant vector. For simplicity, we assume \(\boldsymbol{\beta}=(1,\dots,1)\), since the extension to a general constant
\(\boldsymbol{\beta}\) is direct.

\subsection{Formulation}

Unlike standard DG methods, the CDG method employs two overlapping meshes: a primal mesh \( \Gamma_h^C \) and a dual mesh \( \Gamma_h^D \). The corresponding finite element spaces are defined as
\begin{align*}
	V_h &:= \{ v \in L^2(\Omega) : v|_K \in  \mathbb{Q}^k(K),\quad \forall K \in \Gamma_h^C\}, \\
	W_h &:= \{ v \in L^2(\Omega) : v|_K \in \mathbb{Q}^k(K),\quad \forall K \in \Gamma_h^D\}.
\end{align*}
For \eqref{equ:conservation_law}, the semidiscrete CDG solution is denoted by \( \boldsymbol{u}_h = (u_h^C, u_h^D) \in V_h \times W_h \). For any test function $\bm{\phi} = (\phi_1, \phi_2) \in V_h \times W_h$, $\boldsymbol{u}_h$ satisfies the system
\begin{equation}\label{equ:semiCDG}
	\begin{aligned}
		(\partial_t u_h^C, \phi_1)_C &= \hat{B}_C(\boldsymbol{u}_h, \phi_1), \quad \forall C \in \Gamma_h^C, \\
		(\partial_t u_h^D, \phi_2)_D &= \tilde{B}_D(\boldsymbol{u}_h, \phi_2), \quad \forall D \in \Gamma_h^D,
	\end{aligned}
\end{equation}
where \( (\cdot, \cdot)_K \) denotes the standard \( L^2 \) inner product over cell \( K \). Unless otherwise stated, \(\bm w=(w_C,w_D)\) denotes an arbitrary element of \(V_h\times W_h\).

\subsection*{1D formulation}
In 1D, the overlapping uniform meshes are defined as
\begin{align*}
\Gamma_h^C = \bigcup_i \{I_i = [x_{i-1/2}, x_{i+1/2}]\}, \quad \Gamma_h^D = \bigcup_i \{I_{i+1/2} = [x_i, x_{i+1}]\},
\end{align*}
with uniform mesh size $h = x_{i+1/2} - x_{i-1/2}$ and center of $I_i$ $x_i = \tfrac{1}{2}(x_{i+1/2} + x_{i-1/2})$. The local spatial operators in \eqref{equ:semiCDG} are defined as
\begin{equation}\label{eq:1D_spatial}
	\begin{aligned}
		\hat{B}_{i}(\bm{w}, \phi_1) &= \frac{1}{\tau_{\max}} \int_{I_i} (w_D - w_C)\phi_1 \,\mathrm{d}x \\
		&\quad + \int_{I_i} w_D \partial_x \phi_1 \,\mathrm{d}x - (w_D)_{i+1/2} (\phi_1)_{i+1/2}^- + (w_D)_{i-1/2} (\phi_1)_{i-1/2}^+, \\
		\tilde{B}_{i+1/2}(\bm{w}, \phi_2) &= \frac{1}{\tau_{\max}} \int_{I_{i+1/2}} (w_C - w_D)\phi_2 \,\mathrm{d}x \\
		&\quad + \int_{I_{i+1/2}} w_C \partial_x \phi_2 \,\mathrm{d}x - (w_C)_{i+1} (\phi_2)_{i+1}^- + (w_C)_{i} (\phi_2)_{i}^+,
	\end{aligned}
\end{equation}
where $\tau_{\max} = C_{\text{CFL}} h$ with $C_{\rm CFL}$ denoting the  CFL number. We adopt standard trace notation: $(w)_{i+1/2}^{\pm} = \lim_{\epsilon \to 0^+} w(x_{i+1/2} \pm \epsilon)$.  The operators $\hat{B}$ and $\tilde{B}$ represent the CDG  spatial operator on the primal and dual meshes, respectively.  Note that the $\frac{\pm1}{\tau_{\max}}(w_D - w_C)$ term couples the two meshes.

\subsection*{2D formulation}
In 2D, the meshes are given by
\begin{align*}
	\Gamma_h^C &= \bigcup_{i,j} \{C_{i,j} = I_i \times J_j\}, \quad \text{with } I_i := [x_{i-1/2}, x_{i+1/2}], \, J_j := [y_{j-1/2}, y_{j+1/2}], \\
	\Gamma_h^D &= \bigcup_{i,j} \{C_{i+1/2, j+1/2} = I_{i+1/2} \times J_{j+1/2}\}, \quad \text{with } I_{i+1/2} := [x_i, x_{i+1}], \, J_{j+1/2} := [y_j, y_{j+1}],
\end{align*}
with $h_x = \max_i |I_i|$, $h_y = \max_j |J_j|$, and $h = \max\{h_x, h_y\}$. The 2D spatial operators are defined as
\begin{subequations}\label{eq:2D_spatial}
	\begin{align}
		\hat{B}_{i,j}(\bm{w}, \phi_1)
		&= \frac{1}{\tau_{\max}} \int_{C_{i,j}} (w_D - w_C) \phi_1 \,\mathrm{d}x\mathrm{d}y
		+ \int_{C_{i,j}} w_D \boldsymbol{\beta} \cdot \nabla \phi_1 \,\mathrm{d}x\mathrm{d}y \notag \\
		&\quad - \int_{J_j} \left[ \beta_1 w_D(x_{i+1/2}, y) (\phi_1)_{i+1/2}^- - \beta_1 w_D(x_{i-1/2}, y) (\phi_1)_{i-1/2}^+ \right] \,\mathrm{d}y \notag \\
		&\quad - \int_{I_i} \left[ \beta_2 w_D(x, y_{j+1/2}) (\phi_1)_{j+1/2}^- - \beta_2 w_D(x, y_{j-1/2}) (\phi_1)_{j-1/2}^+ \right] \,\mathrm{d}x,
		\\
		\tilde{B}_{i+1/2, j+1/2}(\bm{w},\phi_2)
		&= \frac{1}{\tau_{\max}} \int_{C_{i+1/2, j+1/2}} (w_C - w_D) \phi_2 \,\mathrm{d}x\mathrm{d}y
		+ \int_{C_{i+1/2, j+1/2}} w_C \boldsymbol{\beta} \cdot \nabla \phi_2 \,\mathrm{d}x\mathrm{d}y \notag \\
		&\quad - \int_{J_{j+1/2}} \left[ \beta_1 w_C(x_{i+1}, y) (\phi_2)_{i+1}^- - \beta_1 w_C(x_i, y) (\phi_2)_i^+ \right] \,\mathrm{d}y \notag \\
		&\quad - \int_{I_{i+1/2}} \left[ \beta_2 w_C(x, y_{j+1}) (\phi_2)_{j+1}^- - \beta_2 w_C(x, y_j) (\phi_2)_j^+ \right] \,\mathrm{d}x.
	\end{align}
\end{subequations}
Here, $\tau_{\max} = \frac{C_{\rm CFL}}{1/h_x + 1/h_y}$. The interface traces at cell boundaries are defined as
\[
(w_C)_{{i+1/2}}^\pm := \lim_{\epsilon \to 0^+} w_C(x_{i+1/2} \pm \epsilon, \cdot), \quad
(w_C)_{{j+1/2}}^\pm := \lim_{\epsilon \to 0^+} w_C(\cdot, y_{j+1/2} \pm \epsilon),
\]
and analogously for \( (w_D)_{i}^\pm \) and \( (w_D)_{j}^\pm \).

\subsection{Existing properties of the CDG method}

The CDG method retains important properties from classical DG methods, including $L^2$-stability, weak boundedness, and optimal convergence. We begin by recalling the stability result of Liu \textit{et al.}~\cite{liu2008l2}. Throughout the paper, for \(\bm w=(w_C,w_D)\), we set
\[
\|\bm w\|_K^2:=\int_K\bigl(w_C^2+w_D^2\bigr)\,dx,
\qquad
\|\bm w\|_\Gamma^2:=\sum_{K\in\Gamma_h^C}\|w_C\|_{\partial K}^2+\sum_{K\in\Gamma_h^D}\|w_D\|_{\partial K}^2.
\]

\begin{proposition}[{\cite[Theorem~2.1]{liu2008l2}}]\label[proposition]{prop:L2stable}
	The solution \( \bm{u}_h = (u_h^C, u_h^D) \in V_h \times W_h \) for \eqref{equ:conservation_law} of the semidiscrete CDG method defined by \eqref{equ:semiCDG} satisfies
	\[
	\frac{d}{dt} \|\bm{u}_h(t)\|^2 = -\frac{2}{\tau_{\max}} \|u_h^C - u_h^D\|^2 \leq 0.
	\]
\end{proposition}

The CDG spatial operator satisfies the weak boundedness property stated in \cref{lem:weaklybdd}.

\begin{lemma}[{\cite[Lemma~3.3]{peng2025oscillation}}]\label[lemma]{lem:weaklybdd}
	For any \(\bm{w}, \bm{\phi}\in V_h\times W_h\),
	\begin{equation}\label{eq:H_bounded}
		\biggl|\sum_{C\in\Gamma_h^C}\hat{B}_C(\bm{w},\phi_1)\biggr| \lesssim\,h^{-1}\,\|\bm{w}\|\,\|\phi_1\|,
		\quad
		\biggl|\sum_{D\in\Gamma_h^D}\tilde{B}_D(\bm{w},\phi_2)\biggr| \lesssim\,h^{-1}\,\|\bm{w}\|\,\|\phi_2\|.
	\end{equation}
\end{lemma}

For the linear advection equation \eqref{equ:conservation_law},  the semidiscrete CDG scheme achieves the optimal \(L^{2}\) convergence order under standard smoothness assumptions, provided that the initial data are projected with accuracy \(\mathcal{O}(h^{k+1})\).
This condition is satisfied by the 1D projection defined in \cref{def:projection} and by its tensor-product extension in \cref{def:projection2D}.
While our 2D tensor-product projection is not identical to the choice adopted in \cite{liu2018optimal}, the same optimal error estimate can be derived from Theorem~3.1 in \cite{liu2018optimal} provided the initial error is of order $\mathcal{O}(h^{k+1})$.
\begin{theorem}[Theorem 3.1 in \cite{liu2018optimal}]\label[theorem]{thm:opt_estimate}
	Let \(\bm{u}_h\) be the \(\mathbb{Q}^k\)-based CDG solution of \eqref{equ:conservation_law}. Assume the exact solution $u(\cdot,t)\in H^{k+2}(\Omega)$ and the initial error satisfies
	\[
	\|\bm{u}_h(\cdot,0)-\bm{u}(\cdot,0)\|\lesssim h^{k+1},
	\]
	where \(\bm{u}(\bm{x},t):=(u(\bm{x},t),u(\bm{x},t))\). Then for all \(t\in[0,T]\),
	\[
	\|\bm{u}_h(\cdot,t)-\bm{u}(\cdot,t)\|\lesssim h^{k+1}.
	\]
\end{theorem}
The optimal-error theorem sets the baseline $L^2$ estimate, but it does not explain the pointwise or cell-average superconvergence. The next section therefore studies the corrected residual and identifies the cancellation that remains after primal-dual coupling removes the full Galerkin orthogonality available in standard DG analyses.

Finally, to facilitate the construction of correction functions (see \cref{def:correctfunc} and \cref{def:correctfun2D}), we introduce the orthogonal decomposition in each direction. For the $x$-direction, we set
\[
\begin{aligned}
	V_x^0 \times W_x^0
	&:= \Bigl\{\bm{v}=(v_1,v_2)\in V_h\times W_h:
	\ v_1(\cdot,y)\in \hat{\mathbb{P}}_{x}^{k}(I_i),\ 
	v_2(\cdot,y)\in \hat{\mathbb{P}}_{x}^{k}(I_{i+1/2})\Bigr\},\\[2mm]
	\overline{V}_x \times \overline{W}_x
	&:= \Bigl\{\bm{v}=(v_1,v_2)\in V_h\times W_h:
	\ v_1(\cdot,y)\in \mathbb{P}_{x}^{0}(I_i),\ 
	v_2(\cdot,y)\in \mathbb{P}_{x}^{0}(I_{i+1/2})\Bigr\},
\end{aligned}
\]
where
$
\hat{\mathbb{P}}_{x}^{k}(I) := \Bigl\{p\in\mathbb{P}^k(I):\int_I p\,\mathrm{d}x=0\Bigr\}$ and 
$\mathbb{P}_{x}^{0}(I) := \mathbb{P}^0(I).
$
Analogous definitions apply in the \(y\)-direction by replacing \(x\), \(I_i\), and \(I_{i+1/2}\) with \(y\), \(J_j\), and \(J_{j+1/2}\), respectively.

\section{1D superconvergent analysis}\label{sec:1Danalysis}
This section proves the 1D superconvergent estimate obtained from the corrected-error framework. The correction functions are introduced in the error equation to cancel the leading projector-induced primal-dual residuals and to expose the constant-test residual that can be estimated sharply. The corrected initialization then reveals the desired rate via the Gr\"onwall inequality. We first state the theorem and its corollary, and then introduce the projector $\bm{P}^*u$ (\Cref{def:projection}) and the $l$-th correction function $\bm{R}^l u$ (\Cref{def:correctfunc}). Throughout this section, the uncorrected and corrected errors are
\[
\bm{\zeta}:=\bm{u}_h-\bm{P}^*u,
\qquad
\hat{\bm{\zeta}}:=\bm{\zeta}+\sum_{l=1}^{\min\{k,2\}}\bm{R}^l u.
\]
\begin{theorem}[Corrected-initialization superconvergence of the 1D semidiscrete CDG method]
	\label[theorem]{thm:supercon1D}
	Let $k \ge 1$ and $\bm{u}_h$ be the semidiscrete CDG approximation  of \eqref{equ:conservation_law}. Assume $u \in C^1([0,T]; {H^{\sigma_k+1}(\Omega)})$ with $T>0$, where $\sigma_k= \min\{2k+1, k+3\}$. If the initial corrected error satisfies $\|\hat{\bm{\zeta}}(\cdot, 0)\| \lesssim h^{\sigma_k}$, then
	\begin{equation}
		\|\hat{\bm{\zeta}}(\cdot, T)\| \lesssim h^{\sigma_k}.
	\end{equation}
\end{theorem}
The proof of \cref{thm:supercon1D} is organized around a projection--correction construction. Its key ingredient is the asymptotic weak cancellation structure of projection or correction under the overlapping-mesh CDG operator. We divide the argument into four steps; Step 3 is where the present proof departs from both the CDG optimal $L^2$ analysis and the standard DG superconvergent argument.
\begin{description}
	\item[Step 1:] As in the standard analysis, we derive the local error equation from \eqref{equ:semiCDG} and define the error decomposition
	\begin{equation}\label{equ:deferror}
		\bm{e} := \bm{u}_h - \bm{u}, \qquad
		\bm{\eta} := \bm{u} - \bm{P}^\ast u.
	\end{equation}
	
	\item[Step 2:] We construct the correction functions \(\bm{R}^l u = (R_C^l u, R_D^l u) \in V_x^0 \times W_x^0\) (see \cref{def:correctfunc}) for \(1\leq l\leq \min\{k,2\}\), and prove the bound
	\(
	\|\bm{R}^l u\| \lesssim h^{k+1+l}
	\) 
	(\cref{prop:correctestimate}). 
	
	\item[Step 3:] We rewrite the error equation in terms of the corrected errors \(\hat{\bm{\zeta}}\) and
	\(
	\hat{\bm{\eta}} := \sum_{l=0}^{\min\{k,2\}} \bm{R}^l u.
	\)
	These correction functions capture the primal and dual projection residuals of the CDG spatial operator for nonconstant test functions. Because the correction operators \(\hat{P}_i\) in \eqref{equ:defproject} and \(\tilde{P}_{i+\frac12}\) vanish for locally constant test functions (see \Cref{rem:constanttest}), we split the test space as \(V_x^0\oplus\overline V_x\). We then exploit the (asymptotic) weak cancellation for the piecewise-constant component in \(\mathcal C_i^2(\phi_1)\), applying \Cref{prop:testfuncest} to the constant part \(\phi_1\in\overline V_x\), and bound the remaining term through the residual operator \(\widehat{\mathcal T}_i\) in \eqref{equ:defT} via \Cref{prop:correct_testfuncest}. The remaining correction terms are of higher order due to the difference structure established in \Cref{coro:overlap}.
	
	\item[Step 4:] Combining the preceding estimates with the stability property in \cref{prop:L2stable}, we obtain
	\(
	\frac{d}{dt}\|\hat{\bm{\zeta}}\|^2 \lesssim h^{\sigma_k}\|\hat{\bm{\zeta}}\|,
	\)
	and the desired superconvergent estimate follows from Gr\"onwall's inequality.
\end{description}

\begin{corollary}[Cell-average superconvergence in 1D]\label[corollary]{coro:supercon1D}
	Suppose the scheme is initialized with \(\bm u_h(\cdot,0)=\bm P^\ast u(\cdot,0)-\bm R^{1}u(\cdot,0)\), and assume that $u$ satisfies the hypotheses of \cref{thm:supercon1D}. Then
	\[ e_{\mathrm{avg}}:=\left(h\sum_i\left(\frac{1}{h}\int_{I_i}(u-u_h^C)\,dx\right)^2\right)^{1/2}\lesssim h^{\sigma_k}.
	\]
\end{corollary}
\begin{proof}
	Let \(\zeta_C:=u_h^C-P_C^\ast u\) and
	\(
	\hat{\zeta}_C:=\zeta_C+\sum_{l=1}^{\min\{k,2\}}R_C^l u.
	\)
Since $P_C^\ast$ preserves cell averages and each correction function has zero cell averages,
	\(\int_{I_i}(u-u_h^C)\,dx=-\int_{I_i}\hat{\zeta}_C\,dx\) for all \(I_i.\)
	Therefore,
	\(
	e_{\mathrm{avg}}
	=\Bigg(h\sum_i\Big(\frac1{|I_i|}\int_{I_i}\hat{\zeta}_C\,dx\Big)^2\Bigg)^{1/2}
	\le \|\hat{\zeta}_C(\cdot,T)\|,
	\)
	and the cell-average estimate follows from \cref{thm:supercon1D} together with \cref{prop:correctestimate}.
\end{proof}

\begin{remark}[Order degradation via inverse inequalities]
In 1D, under corrected initialization and for \(k\ge2\), inverse inequalities recover the \(h^{k+2}\) \(\ell^\infty\)-rate at the superconvergent points from the corrected-error estimate in \Cref{thm:supercon1D}. This route, however, is not sufficient when initialized by the LSZ projection, where the corrected error is only of order \(h^{k+2}\) and a direct inverse inequality would lose a factor \(h^{-1/2}\). The LSZ-projection-based pointwise \(\ell^\infty\) superconvergence theorem in \Cref{thm:base_LSZ_maximum} therefore uses a different final step: it controls difference quotients of the sampled pointwise errors and then applies a discrete Sobolev inequality on the set of superconvergent points.
\end{remark}

\begin{remark}[Rate limitation for $k=1$]\label[remark]{rem:pc-bound}
For \(k=1\), the finite element space \(V_h\times W_h\) only supports linear correction functions. Since \Cref{def:correctfunc} involves no quadratic test functions, a second correction level of order \(h^{k+3}\) cannot be constructed for \(\mathbb{P}^1\) elements. As a result, the superconvergence rate for \(k=1\) is limited to that stated in \Cref{coro:supercon1D}.
\end{remark}

\subsection{Projection and its weak cancellation}

In 1D, we adopt the projection operators $P_C^*$ and $P_D^*$ as introduced in~\cite{liu2018optimal}. The projection on the primal mesh $P_C^*$ maps into $V_h$, while the dual projection $P_D^*$ is defined analogously for $W_h$. Although these projections are local and can be defined on the reference cell \([-1,1]\), we formulate them on each physical cell, since this notation makes the overlap shift more transparent and better reveals the difference structure used later.
\begin{definition} \label[definition]{def:projection}
	For \( \omega \in {W^{1,\infty}(\Omega)} \), the projection $P_C^*\omega \in V_h$ satisfies
	\begin{equation}\label{equ:defproject}
		\left\{\begin{aligned}
			\int_{I_i} P_C^* \omega(x)\, \mathrm{d} x &= \int_{I_i} \omega(x)\, \mathrm{d} x,\\
			\widehat{P}_i\left(P_C^* \omega ; \phi\right) &= \widehat{P}_i\left(\omega ; \phi\right), \quad \forall \phi \in \mathbb{P}^{k}(I_i).
		\end{aligned}\right.
	\end{equation}
	Following \cite{liu2018optimal},
	\begin{equation*}
		\resizebox{0.99\hsize}{!}{$
			\begin{aligned}
				\widehat{P}_i(\omega ; \phi) &:= \frac{1}{\tau_{\max}} \left[ \int_{x_{i-1/2}}^{x_i} (\omega(x+h/2)-\omega(x))\phi(x) \,\mathrm{d}x + \int_{x_i}^{x_{i+1/2}} (\omega(x-h/2)-\omega(x))\phi(x) \,\mathrm{d}x \right] \\
				&\quad + \int_{x_{i-1/2}}^{x_i} \omega(x+h/2)\partial_x\phi \,\mathrm{d}x + \int_{x_i}^{x_{i+1/2}} \omega(x-h/2)\partial_x\phi \,\mathrm{d}x \\
				&\quad - \omega(x_i) \left(\phi(x_{i+1/2}^{-}) - \phi(x_{i-1/2}^{+})\right).
			\end{aligned}$}
	\end{equation*}
\end{definition}
\begin{remark}\label[remark]{rem:constanttest}
	It may appear that the system in \cref{def:projection} involves $k+1$ unknowns subject to $k+2$ constraints. However, for a constant test function $\phi$, the variational formulation reduces to the trivial identity $0=0$. Consequently, the system consists of only $k+1$ constraints.
\end{remark}
\begin{proposition}\label[proposition]{prop:roots}
	For any cell \(I_i\) and any polynomial $\hat{u} \in \mathbb{P}^{k+1}(I_i)$, $\hat{u}-P_C^*\hat{u}$ has at least one root in \(I_i\).
\end{proposition}

\begin{proof}
	Since the projection $P_C^*$ acts as the identity on $\mathbb{P}^k(I_i)$, it suffices to prove the claim for $v(x) := x^{k+1}$. Consider the residual $e(x) := v(x) - P_C^* v(x)$ and define its primitive
	\(
	F(x) := \int_{x_{i-1/2}}^{x} e(s) \,\mathrm{d}s.
	\)
	Clearly, \(F(x_{i-1/2})=0\). And since $P_C^*$ preserves cell averages, we have
	\(
	F(x_{i+1/2}) = \int_{I_i} (v(s) - P_C^* v(s)) \,\mathrm{d}s = 0.
	\)
	By Rolle's theorem, there exists \(s_0 \in I_i\) such that \(F'(s_0) = e(s_0) = 0\), which completes the proof.
\end{proof}

The projection satisfies the standard approximation properties stated in \cref{lem:standard_app}.

\begin{lemma}\label[lemma]{lem:standard_app}
	The projection $\bm{P}^* \in \mathcal{L}(H^{k+1}(\Omega), L^2(\Omega)\times L^2(\Omega) )$ (\cref{def:projection}) is well-defined and acts as the identity on $V_h\times W_h$. For any $\omega \in H^{k+1}(\Omega)$, letting $\bm{\omega} := (\omega, \omega)$, we have
	\(
	\|\bm{P}^*\omega - \bm{\omega}\| + h\|\bm{P}^*\omega - \bm{\omega}\|_{\infty} + h^{1/2}\|\bm{P}^*\omega - \bm{\omega}\|_{\Gamma} \lesssim h^{k+1}\|\omega\|_{H^{k+1}(\Omega)}.
	\)
\end{lemma}

As remarked in \cite{yang_shu_2012}, for highly oscillatory or discontinuous functions, the standard approximation is no longer practical. Nevertheless, the projection remains bounded in $L^\infty$ norm.

\begin{lemma}[\(L^\infty\)-boundedness, Lemma 2.1 in \cite{liu2018optimal}]\label[lemma]{lem:linfproj}
	For the projection in \cref{def:projection} and any $w \in W^{1,\infty}(\Omega)$, we have
	\[
	\|P_C^* w\|_{L^{\infty}(I_i)} \lesssim\|w\|_{L^{\infty}(I_i)}, \quad \|P_D^* w\|_{L^{\infty}(I_{i+1/2})} \lesssim\|w\|_{L^{\infty}(I_{i+1/2})}, \quad \text{for any } i.
	\]
\end{lemma}

From \cref{def:projection}, we further establish the following results.

\begin{proposition}\label[proposition]{lem:projectlem}
	For any $\bm{F} = (F_1, F_2) \in V_x^0 \times W_x^0$, there exists a unique $\bm{w} \in V_h \times W_h$ such that
	\begin{equation}\label{equ:projectlem}
	\left\{\begin{aligned}
		\widehat{P}_i(w_C, \phi_1) &= (F_1, \phi_1)_i, \quad \forall \phi_1 \in V_x^0, \\
		(w_C, 1)_i &= 0.
	\end{aligned}\right.
	\end{equation}
	The component $w_D$ is similarly defined on the dual mesh. Then,
	\[
	\|w_C\|_{L^{\infty}(I_i)} \lesssim h^{1/2} \|F_1\|_{i}, \quad \|w_D\|_{L^{\infty}(I_{i+1/2})} \lesssim h^{1/2} \|F_2\|_{i+1/2}.
	\]
\end{proposition}

\begin{proof}	
		As shown in \cite[Appendix A.1]{liu2018optimal}, testing with the standard Legendre polynomial basis $\{L_p^i(x) : 0 \le p \le k\}$ for $\mathbb{P}^k(I_i)$, we can rewrite \eqref{equ:projectlem}  into \[
		A_i^T \bm{b}(w_C) = \bm{f}, \quad \text{with} \quad f_p = (F_1, L_p^i)_{i},
		\] where $\bm{b}(w_C)$ is the Legendre coefficients vector of $w_C$ and $A_i^T$ is invertible with a uniformly bounded inverse. 
	Using the boundedness of $(A_i^T)^{-1}$ and the standard inverse estimate for Legendre expansions, we obtain
	\[
	\|w_C\|_{L^\infty(I_i)} \lesssim \|\bm{b}(w_C)\|_{\ell^\infty} \le \|(A_i^T)^{-1}\|_{\ell^\infty} \|\bm{f}\|_{\ell^\infty} \lesssim h^{1/2} \|F_1\|_{i}.
	\]
	The estimate for $w_D$ on the dual mesh follows analogously.
\end{proof}

The projection also satisfies the following superconvergent properties under the spatial operators (\Cref{lem:supercon_project}) and the weak cancellation (\Cref{prop:testfuncest}).

\begin{lemma}\label[lemma]{lem:supercon_project}
	For any cell $I_i \in \Gamma_h^C$ and $I_{i+1/2} \in \Gamma_h^D$, the projections satisfy
	\begin{align*}
		|\hat{B}_i(\bm{\eta}, \phi_1)| &\lesssim h^{k+1} \|u\|_{H^{k+2}(I_i)} \|\phi_1\|_{i}, \\
		|\tilde{B}_{i+1/2}(\bm{\eta}, \phi_2)| &\lesssim h^{k+1} \|u\|_{H^{k+2}(I_{i+1/2})} \|\phi_2\|_{i+1/2}.
	\end{align*}
\end{lemma}

This result, first proved in~\cite{liu2018optimal}, is the key ingredient for optimal error estimates. In contrast, general projection operators typically only yield \( \mathcal{O}(h^{k+1/2}) \) suboptimal accuracy. Moreover, \cref{prop:testfuncest} establishes an additional weak cancellation for piecewise-constant test functions.

\begin{lemma}[Weak cancellation]\label[lemma]{prop:testfuncest}
	For any $I_i \in \Gamma_h^C$ and $I_{i+1/2} \in \Gamma_h^D$, we have
	\begin{equation}\label{equ:constant_test}
		\tilde{B}_{i+1/2}(\bm{\eta}, 1) = 0, \quad \hat{B}_i(\bm{\eta}, 1) = 0.
	\end{equation}
\end{lemma}

\begin{proof}
	Applying a change of variables that shifts $\eta_D(x)$ by $h/2$ to align with $\phi_2$ terms, we have
	\begin{equation}\label{eq:Ptilde-shift}
		\begin{aligned}
			\widetilde{P}_{i+1/2}(\eta_D; \phi_2) &= \frac{1}{\tau_{\max}} \left( \int_{x_i}^{x_{i+1/2}} \eta_D [\phi_2(x+h/2) - \phi_2(x)] \,\mathrm{d}x \right. \\
			&\quad \left. + \int_{x_{i+1/2}}^{x_{i+1}} \eta_D [\phi_2(x-h/2) - \phi_2(x)] \,\mathrm{d}x \right) \\
			&\quad + \int_{x_i}^{x_{i+1/2}} \eta_D \partial_x \phi_2(x+h/2) \,\mathrm{d}x + \int_{x_{i+1/2}}^{x_{i+1}} \eta_D \partial_x \phi_2(x-h/2) \,\mathrm{d}x \\
			&\quad - (\eta_D)_{i+1/2} \big[ (\phi_2)_{i+1} - (\phi_2)_i \big].
		\end{aligned}
	\end{equation}
	In 1D, selecting ${\phi}_2 = \frac{x-x_{i+1/2}}{h/2}$ in $\widetilde{P}_{i+1/2}(\eta_D, \phi_2)$ and using that $P_D^*$ preserves the cell average on $I_{i+\frac12}$, we obtain 
	\[    	\resizebox{0.99\hsize}{!}{$
		(\eta_D)_{i+1/2} = \frac{1}{2\tau_{\max}} \left[ {\int_{x_i}^{x_{i+1/2}}} \eta_D {\mathrm{d}x}-  { \int_{x_{i+1/2}}^{x_{i+1}}} \eta_D {\mathrm{d}x}\right] = \frac{1}{\tau_{\max}} {\int_{x_i}^{x_{i+1/2}}} \eta_D{\mathrm{d}x} = \frac{-1}{\tau_{\max}} {\int_{x_{i+1/2}}^{x_{i+1}}} \eta_D{\mathrm{d}x}.$}
	\]
	Therefore,
	\[
	\hat{B}_i(\bm{\eta}, 1) = \frac{1}{\tau_{\max}} \left[ {\int_{x_{i-1/2}}^{x_{i}}} \eta_D {\mathrm{d}x}+ {\int_{x_{i}}^{x_{i+1/2}}} \eta_D{\mathrm{d}x} \right] - (\eta_D)_{i+1/2} + (\eta_D)_{i-1/2} = 0.
	\]
	Similarly, by selecting $\phi_1 = \frac{x-x_i}{h/2}$ in $\widehat{P}_i(\eta_C,\phi_1)$, we have $\tilde{B}_{i+\frac12}(\bm{\eta},1) =0$.
\end{proof}

\begin{remark}[Weak cancellation versus DG Galerkin orthogonality]
	The weak cancellation $\hat{B}_i(\bm{\eta},1)=0$ in \Cref{prop:testfuncest} is crucial for the 1D analysis. In the test space decomposition of \eqref{equ:C2}, the contribution from the constant-test component cannot be removed by the correction alone. This identity eliminates the $\bm{\eta}$-term from this component. Without this weak cancellation, the surviving lower-order term would limit the final superconvergent estimate. In this sense, it plays the role of a weaker replacement for the full Galerkin orthogonality available in standard DG methods \cite{cao_zhang_zou_2014}, where $\hat{H}_i(\eta,v)=0$ for all $v\in\mathbb{P}^k(I_i)$. In 2D, applying the 1D weak cancellation to a single directional projection or correction residual is not sufficient to obtain the numerically observed optimal superconvergence residual order. The required order is recovered through the HOCA mechanism in \Cref{prop:correct_testfuncest2D}, after aggregating the directional projection residual with the corresponding correction residuals.
\end{remark}
\subsection{Correction functions and their high-order weak cancellation}
This subsection defines and analyzes the correction functions within the projection--correction framework. These functions play a central role in the superconvergent analysis by eliminating the lower-order terms left by the initial projection, while themselves being of sufficiently high order to preserve the optimal superconvergence rate. Just as importantly, we explain why the CDG correction hierarchy stops after two levels, rather than continuing to the full chain typical of standard DG methods. 

\begin{definition}[Correction functions]\label[definition]{def:correctfunc}
	We set $\bm R^0 u := \bm\eta$.
	For $l=1,2$, define $\bm R^l u=(R_C^l u, R_D^l u)\in V_x^0\times W_x^0$ by requiring that, for all $\bm \phi=(\phi_1,\phi_2)\in V_x^0\times W_x^0$,
	\begin{equation*}
		\begin{aligned}
			\widehat{P}_i(R_C^l u, \phi_1) &= -(R_C^{l-1}\partial_x u, \phi_1)_i - \hat{B}_i(\bm{R}^{l-1}u, \phi_1) + \widehat{P}_i(R_C^{l-1}u, \phi_1), \\
			\widetilde{P}_{i+1/2}(R_D^l u, \phi_2) &= -(R_D^{l-1}\partial_x u, \phi_2)_{i+1/2} - \tilde{B}_{i+1/2}(\bm{R}^{l-1}u, \phi_2) + \widetilde{P}_{i+1/2}(R_D^{l-1}u, \phi_2).
		\end{aligned}
	\end{equation*}
\end{definition}
\begin{remark}[Why two correction levels are sufficient]
\label[remark]{rem:two-level-correction}
Higher-level correction functions can still be defined for \(k\ge2\). However, they do not improve the final superconvergent estimate. The limiting contribution in the CDG error analysis is the order of the asymptotic weak cancellation, rather than the order of the remaining correction residuals. In the present CDG framework, this cancellation yields at most the rate \(k+3\). Therefore, further correction levels do not change the final convergence order. This differs from the standard DG analysis, where the full Galerkin orthogonality leads to a cancellation rate of \(2k+1\).
\end{remark}

By \cref{lem:projectlem}, the correction \(\bm{R}^{l}u\) is well-defined. Due to the change of integration intervals, the projections and corrections satisfy the shifting property stated in \cref{lem:shifting}.

\begin{lemma}\label[lemma]{lem:shifting}
	For any function $u\in {W^{1,\infty}(\Omega)}$ and $0\leq l\leq 2$, on a uniform mesh, we have the following shifting identities
	\[
	(R_C^l u)(x \pm h/2) = R_D^l[u(x \pm h/2)],\quad 
	(R_D^l u)(x \pm h/2) = R_C^l[u(x \pm h/2)].
	\]
\end{lemma}

Let $\Delta_{x}^\pm$ denote the finite difference operator
\begin{equation}\label{equ:defdifference}
	\Delta_{x}^\pm u(x) := u(x) - u(x \pm h/2).
\end{equation}
Using the shifting property and inverse inequalities, we directly obtain estimates for the residuals
\begin{subequations}\label{equ:defT}
	\begin{align}
		\widehat{\mathcal{T}}_i(\phi_1; \bm{R}^l u) &:= \hat{B}_i(\bm{R}^l u, \phi_1) - \widehat{P}_i(R_C^l u, \phi_1), \\
		\widetilde{\mathcal{T}}_{i+1/2}(\phi_2; \bm{R}^l u) &:= \tilde{B}_{i+1/2}(\bm{R}^l u, \phi_2) - \widetilde{P}_{i+1/2}(R_D^l u, \phi_2).
	\end{align}
\end{subequations}

\begin{corollary}\label[corollary]{coro:overlap}
	For all $\bm{\phi} \in V_h \times W_h$ and $0 \leq l \leq \min\{k,2\}$, the shifting property implies a finite difference structure, yielding the estimates
	\[
	\biggl| \widehat{\mathcal{T}}_i(\phi_1; \bm{R}^l u)\biggr| \lesssim h^{-1} \|R_D^l( \Delta_x^\pm u)\|_i \|\phi_1\|_i,
	\]	\[
	\biggl| \widetilde{\mathcal{T}}_{i+1/2}(\phi_2; \bm{R}^l u)\biggr| \lesssim h^{-1} \|R_C^l( \Delta_x^\pm u)\|_{i+1/2} \|\phi_2\|_{i+1/2}.
	\]
\end{corollary}
\begin{proof}
	It suffices to prove the estimate for \(\widetilde{\mathcal T}_{i+1/2}(\phi_2;\bm R^l u)\), since the estimate for \(\widehat{\mathcal T}_i(\phi_1;\bm R^l u)\) is obtained in the same way on the primal mesh after interchanging the roles of \(C\) and \(D\). By the definition of \(\widetilde{\mathcal T}_{i+1/2}\), substituting \eqref{eq:1D_spatial} and the definition of \(\widetilde P_{i+\frac12}\), all terms containing \((R_D^l u)(x)\) cancel. Using \Cref{lem:shifting}, we obtain
	\[
	\begin{aligned}
		\widetilde{\mathcal T}_{i+1/2}(\phi_2;\bm R^l u)
		&=
		\int_{x_i}^{x_{i+1/2}} R_C^l(\Delta_x^+u)\Bigl(\frac{1}{\tau_{\max}}\phi_2+\partial_x\phi_2\Bigr)\,\mathrm{d}x
		-[R_C^l(\Delta_x^+u)]_{i+1}(\phi_2)_{i+1}
		\\
		&\quad+
		\int_{x_{i+1/2}}^{x_{i+1}} R_C^l(\Delta_x^-u)\Bigl(\frac{1}{\tau_{\max}}\phi_2+\partial_x\phi_2\Bigr)\,\mathrm{d}x
		+[R_C^l(\Delta_x^-u)]_i(\phi_2)_i .
	\end{aligned}
	\]
	Since \(\tau_{\max}=\mathcal O(h)\), the Cauchy--Schwarz inequality and the inverse inequality yield
	\[
	\bigl|\widetilde{\mathcal T}_{i+1/2}(\phi_2;\bm R^l u)\bigr|
	\lesssim
	h^{-1}\|R_C^l(\Delta_x^\pm u)\|_{i+1/2}\,\|\phi_2\|_{i+1/2}.
	\]
	This proves the second estimate. The first estimate follows analogously on the primal mesh.
\end{proof}
As a consequence, we obtain the following regularity estimates for the corrections.

\begin{proposition}\label[proposition]{prop:correctestimate}
	Let \(0\le l\le \min\{k,2\}\), and let \(v\in H^{k+1+l}(\Omega)\). Then
	\[
	\|R_C^{l}v\| \lesssim h^{k+1+l}\,\|v\|_{H^{k+1+l}(\Omega)},
	\qquad
	\|R_D^{l}v\| \lesssim h^{k+1+l}\,\|v\|_{H^{k+1+l}(\Omega)}.
	\]
\end{proposition}
\begin{proof}
	We proceed by induction on \(l\). For \(l=0\), since \(\bm R^0v=v-\bm P^\ast v\), the conclusion follows immediately from the standard approximation property of \(\bm P^\ast\).
	
	Now let \(1\le l\le \min\{k,2\}\), and assume that the estimate holds for \(l-1\). By \cref{def:correctfunc} and the definition of \(\widehat{\mathcal T}_i\) in \eqref{equ:defT}, we have
	\[
	\widehat P_i(R_C^l v,\phi_1)
	=
	-(R_C^{l-1}\partial_x v,\phi_1)_i
	-
	\widehat{\mathcal T}_i(\phi_1;\bm R^{l-1}v),
	\qquad \forall \phi_1\in V_x^0.
	\]
	Applying \cref{lem:projectlem}, and then summing over all primal cells with the difference structure from \cref{coro:overlap}, we obtain
	\begin{equation}\label{eq:RC-global-recursion-new}
		\|R_C^l v\|
		\lesssim
		h\|R_C^{l-1}\partial_x v\|
		+
		\|R_D^{l-1}(\Delta_x^\pm v)\|.
	\end{equation}
	An analogous argument on the dual mesh yields
	\begin{equation}\label{eq:RD-global-recursion-new}
		\|R_D^l v\|
		\lesssim
		h\|R_D^{l-1}\partial_x v\|
		+
		\|R_C^{l-1}(\Delta_x^\pm v)\|.
	\end{equation}
	
	By the induction hypothesis,
	\(
	h\|R_C^{l-1}\partial_x v\|
	+
	h\|R_D^{l-1}\partial_x v\|
	\lesssim
	h^{k+1+l}\|v\|_{H^{k+1+l}(\Omega)}.
	\)
	Moreover, using
	\(
	\|\Delta_x^\pm v\|_{H^m(\Omega)}\lesssim h\|v\|_{H^{m+1}(\Omega)}\) for \(m\ge0,
	\)
	and applying the induction hypothesis again, we obtain
	\[
	\|R_C^{l-1}(\Delta_x^\pm v)\|
	+
	\|R_D^{l-1}(\Delta_x^\pm v)\|
	\lesssim
	h^{k+1+l}\|v\|_{H^{k+1+l}(\Omega)}.
	\]
	Substituting these bounds into \eqref{eq:RC-global-recursion-new} and \eqref{eq:RD-global-recursion-new} gives
	\[
	\|R_C^l v\|
	\lesssim
	h^{k+1+l}\|v\|_{H^{k+1+l}(\Omega)},
	\qquad
	\|R_D^l v\|
	\lesssim
	h^{k+1+l}\|v\|_{H^{k+1+l}(\Omega)}.
	\]
	This completes the induction.
\end{proof}
For locally constant test functions, the correction terms no longer satisfy an exact weak cancellation, because the full Galerkin orthogonality available in standard DG methods is lost. Instead, only an asymptotic weak cancellation remains. The following theorem shows that the lowest cancellation rate is \(k+3\), which limits the effective number of correction levels, as discussed in \Cref{rem:two-level-correction}.
\begin{theorem}[Asymptotic weak cancellation property]\label[theorem]{prop:correct_testfuncest}
	For any $\bm{\phi} \in \overline{V}_x \times \overline{W}_x$ and  \(l\in\{1,\dots,\min\{k,2\}\}\), the $l$-th correction satisfies
	\begin{subequations}\label{equ:correct_constant_test}
		\begin{align}
			\biggl|\sum_i \tilde{B}_{i+1/2}(\bm{R}^{l}u, \phi_2)\biggr| &\lesssim h^{\sigma_k} \|u\|_{H^{\sigma_k+1}}\,\|\phi_2\|, \\
			\biggl|\sum_i \hat{B}_{i}(\bm{R}^{l}u, \phi_1)\biggr| &\lesssim h^{\sigma_k} \|u\|_{H^{\sigma_k+1}}\,\|\phi_1\|.
		\end{align}
	\end{subequations}
\end{theorem}

\begin{proof}
		When $k=1$, \Cref{coro:overlap} and \Cref{prop:correctestimate} imply
	\[
\begin{aligned}
		\biggl|\sum_i \hat{B}_i(\bm{R}^1 u, \phi_1)\biggr| =	\biggl|\sum_i \phi_1[(R_D^1 u)_{i-\frac12}-(R_D^1 u)_{i+\frac12}]\biggr|\lesssim h^{3} \|\phi_1\| \|u\|_{H^{3}}.
\end{aligned}
	\]
	For $k \ge 2$, on the dual cell $I_{i\pm 1/2}$, consider the linear test function $	\tilde{\phi}_{i\pm\frac12}(x)$ and the quadratic test function $\hat{\phi}_{i\pm\frac12}(x)$
	\[
	\tilde{\phi}_{i\pm\frac12}(x) := \frac{2(x - x_{i\pm 1/2})}{h} \in W_x^0,\quad  \hat{\phi}_{i\pm\frac12}(x) := \left( \frac{2(x - x_{i\pm1/2})}{h} \right)^2 \in \mathbb{P}^2(I_{i\pm1/2}) \subset W_h.
	\]
	
	Consequently, the boundary term in $\widetilde{P}_{i+1/2}(\eta_D, \cdot)$ vanishes since	this function satisfies \[\hat{\phi}_{i+\frac12}(x_{i+1}^{-}) - \hat{\phi}_{i+\frac12}(x_{i}^{+}) = 1 - 1 = 0,\qquad \hat{\phi}_{i+\frac12}(x \pm h/2) - \hat{\phi}_{i+\frac12}(x) = \pm 2 \tilde{\phi}_{i+\frac12} + 1,\]
	and $\hat{\phi}_{i+\frac12}'(x \pm h/2) = \frac{4}{h}  \tilde{\phi}_{i+\frac12}(x \pm h/2).$
	For any $w_D \in L^2(\Omega)$, substituting these into \eqref{eq:Ptilde-shift} and regrouping yields the identity
	\begin{equation}\label{eq:P2Ptilde}
		\resizebox{0.99\hsize}{!}{$
			\begin{aligned}
				\widetilde{P}_{i+1/2}(w_D; \hat{\phi}_{i+\frac12}) &= \frac{1}{\tau_{\max}} \int_{x_i}^{x_{i+1}} w_D \,\mathrm{d}x  + \frac{2}{\tau_{\max}} \left( \int_{x_i}^{x_{i+1/2}}  \tilde{\phi}_{i+\frac12} w_D \,\mathrm{d}x - \int_{x_{i+1/2}}^{x_{i+1}} \tilde{\phi}_{i+\frac12} w_D \,\mathrm{d}x \right) \\
				&\quad + \frac{4}{h} \int_{x_i}^{x_{i+1}}  \tilde{\phi}_{i+\frac12} w_D \,\mathrm{d}x + \frac{4}{h} \left( \int_{x_i}^{x_{i+1/2}} w_D \,\mathrm{d}x - \int_{x_{i+1/2}}^{x_{i+1}} w_D \,\mathrm{d}x \right).
			\end{aligned}$}
	\end{equation}
	(The same formula holds on $I_{i-1/2}$ after an index shift.)
	
	Now set $w_D = (I - P_D^*) \partial_x u$. After shifting the half-cell integrals in \eqref{eq:P2Ptilde}, we obtain
	\[
	\begin{aligned}
		\bigl((I - P_D^*) \partial_xu,  \tilde{\phi}_{i+\frac12}\bigr)_{i+ 1/2} &= \int_{x_i}^{x_{i+1/2}} (I - P_C^*) [\partial_x u(x+\frac{h}{2})] \Bigl( \frac{h}{2\tau_{\max}}  \tilde{\phi}_{i+\frac12} +{ 1} \Bigr) \,\mathrm{d}x \\
		&\quad - \int_{x_{i+1/2}}^{x_{i+1}} (I - P_C^*)[ \partial_x u(x-\frac{h}{2})] \Bigl( \frac{h}{2\tau_{\max}}  \tilde{\phi}_{i+\frac12} + { 1} \Bigr) \,\mathrm{d}x.
	\end{aligned}
	\]
	 A direct computation gives
	\begin{equation}\label{eq:B-interface}
		\hat{B}_i(\bm{R}^1 u, 1) = \frac{1}{2} \Bigl[ \widetilde{P}_{i+1/2}(R_D^1 u, \tilde{\phi}_{i+\frac12}) - \widetilde{P}_{i-1/2}(R_D^1 u,\tilde{\phi}_{i-\frac12}) \Bigr],
	\end{equation}
	with
	\(
	\widetilde{P}_{i\pm 1/2}(R_D^1 u,\tilde{\phi}_{i\pm\frac12}) = -\bigl((I - P_D^*) \partial_x u, \tilde{\phi}_{i\pm\frac12}\bigr)_{i\pm 1/2} - \widetilde{\mathcal{T}}_{i\pm 1/2}{(\tilde{\phi}_{i\pm\frac12};\bm{\eta})}.
	\)
	
	Shifting $I_{i-1/2}$ to $I_{i+1/2}$ in \eqref{eq:B-interface}, and letting $v(x) = \frac{1}{2}(u(x-h) - u(x))$, we further obtain
	\[
	\begin{aligned}
		\hat{B}_i(\bm{R}^1 u, 1) &= \int_{x_i}^{x_{i+1/2}} (I - P_C^*) \Bigl( \partial_x v(x + h/2) + \frac{\Delta_x^+ v}{h/2} \Bigr) \Bigl[ \frac{h}{2\tau_{\max}}\tilde{\phi}_{i+\frac12} + 1 \Bigr] \,\mathrm{d}x \\
		&\quad - \int_{x_{i+1/2}}^{x_{i+1}} (I - P_C^*) \Bigl( \partial_x v(x - h/2) - \frac{\Delta_x^- v}{h/2} \Bigr) \Bigl[ \frac{h}{2\tau_{\max}}\tilde{\phi}_{i+\frac12} + 1 \Bigr] \,\mathrm{d}x \\
		&\quad - \Bigl[ (I - P_C^*) \Delta_x^+ \Delta_x^- v \Bigr]_{i+1}-\Bigl[ (I - P_C^*) \Delta_x^- (v(x)-v(x-h)) \Bigr]_{i+1},
	\end{aligned}
	\]
	where the difference $\Delta_x^\pm v$ is defined in \eqref{equ:defdifference}. 
	By a Taylor expansion, \cref{lem:standard_app}, and inverse inequalities,
	\[
	\left\| (I - P_C^*) \Bigl( \partial_x v(x \pm h/2) \pm \frac{\Delta_x^\pm v}{h/2} \Bigr) \right\| \lesssim h^{k+3} \|u\|_{H^{\sigma_k+1}},
	\]
	and
	\[
	\sum_i\biggl|  \Bigl[ (I - P_C^*) [\Delta_x^+ \Delta_x^- v] \Bigr]_{i+1} \biggr|+ \sum_i\biggl| \Bigl[ (I - P_C^*) \Delta_x^- (v(x)-v(x-h)) \Bigr]_{i+1} \biggr|\lesssim h^{k+3} \|u\|_{{ H^{\sigma_k+1}}} \|1\|.
	\]
	Combining the above  bounds for $\hat{B}_i(\bm{R}^1 u, \phi_1)$ gives 
	$
	\bigl| \sum_i \hat{B}_i(\bm{R}^1 u, \phi_1) \bigr| \lesssim h^{k+3} \|u\|_{{ H^{\sigma_k+1}}} \|\phi_1\|.
	$ 
	The estimate for $\bigl|\sum_i \tilde{B}_{i+1/2}(\bm{R}^1 u, \phi_2)\bigr|$ follows analogously on the dual mesh. The result for $l=2$ follows from \cref{coro:overlap} and \cref{prop:correctestimate}.
\end{proof}
\subsection{Proof of \cref{thm:supercon1D}}

Since both the exact solution $\bm{u}$ and the CDG solution $\bm{u}_h$ satisfy \eqref{equ:semiCDG}, subtracting the two and using the decomposition \eqref{equ:deferror} yields the error equation
\(
(\partial_t \zeta_C, \phi_1)_i - \hat{B}_i(\bm{\zeta}, \phi_1) = (\partial_t \eta_C, \phi_1)_i - \hat{B}_i(\bm{\eta}, \phi_1).
\)
After $m$ correction steps, the corrected error equation is formulated as
\[
(\partial_t \hat{\zeta}_C, \phi_1)_i - \hat{B}_i(\hat{\bm{\zeta}}, \phi_1) = \mathcal{C}_i^1(\phi_1) + \mathcal{C}_i^2(\phi_1) + \mathcal{C}_i^3(\phi_1),
\]
where the residual terms are defined as:
\begin{subequations}\label{eq:Ci-terms}
	\begin{align}
		\mathcal{C}_i^1(\phi_1) &:= (\partial_t R_C^{\min\{k,2\}}u, \phi_1)_i, \label{equ:C1} \\
		\mathcal{C}_i^2(\phi_1) &:= -\sum_{l=1}^{\min\{k,2\}} \Bigl[ \widehat{P}_i(R_C^{l}u, \phi_1) - (\partial_t R_C^{l-1}u, \phi_1)_i + \widehat{\mathcal{T}}_i(\phi_1; \bm{R}^{l-1}u) \Bigr], \label{equ:C2} \\
		\mathcal{C}_i^3(\phi_1) &:= -\widehat{\mathcal{T}}_i(\phi_1; \bm{R}^{\min\{k,2\}}u). \label{equ:C3}
	\end{align}
\end{subequations}

To conclude the proof, it suffices to show that
\(
\sum_i \bigl(|\mathcal{C}_i^1(\phi_1)| + |\mathcal{C}_i^2(\phi_1)| + |\mathcal{C}_i^3(\phi_1)|\bigr) \lesssim h^{\sigma_k} \|\phi_1\|.
\)
Once this estimate is established, the same argument applies on the dual mesh. Taking
\(\bm\phi=\hat{\bm\zeta}\) and using the stability property in \Cref{prop:L2stable}, we obtain
\(
\frac{1}{2}\frac{d}{dt}\|\hat{\bm\zeta}\|^2
\lesssim
h^{\sigma_k}\|\hat{\bm\zeta}\|.
\)
The estimate in \Cref{thm:supercon1D} then follows from Gr\"onwall's inequality, together with the superconvergent initialization bound provided by \Cref{prop:correctestimate}.

By \cref{prop:correctestimate} and \cref{coro:overlap}, we have the bound
\[
\sum_i \bigl(|\mathcal{C}_i^1(\phi_1)| + |\mathcal{C}_i^3(\phi_1)|\bigr) \lesssim h^{\sigma_k} \|\phi_1\|.
\]

\subsubsection*{Estimate of $\sum_i \mathcal{C}_i^2(\phi_1)$:}
We decompose the test function space and use $\partial_t \bm{P}^*u = -\bm{P}^*\partial_xu$ and $\partial_t \bm{R}^l u = -\bm{R}^l\partial_xu$, which follow from \eqref{equ:conservation_law}. For $\phi_1 \in V_x^0$, \cref{def:correctfunc} implies that
\(
\mathcal{C}_i^2(\phi_1) = 0.
\)

For the remaining component $\phi_1 \in \overline{V}_x$, \cref{prop:testfuncest} handles the $\bm{R}^0u=\bm{\eta}$ term and \cref{prop:correct_testfuncest} handles the correction terms. Hence
\[
\sum_i |\mathcal{C}_i^2(\phi_1)|
\lesssim
\sum_i \biggl| \sum_{q=0}^{\min\{k,2\}-1} \widehat{\mathcal{T}}_i(\phi_1; \bm{R}^{q}u) \biggr|
\lesssim h^{\sigma_k} \|\phi_1\|.
\]
Combining these estimates yields the desired result.

\section{Multidimensional superconvergent analysis}\label{sec:2DCDG}

 The extension from 1D to multiple dimensions is nontrivial. In 1D, superconvergence relies on the weak cancellation property from \Cref{prop:testfuncest}. For the tensor product projection
	$
	\bm Q^*=\prod_{i=1}^d \bm P^{x_i},
	$ on Cartesian meshes, the 1D weak cancellation applied to a single directional component does not by itself yield the observed $h^{k+2}$ superconvergence residual order. A different mechanism is therefore required.
	
	Our approach is to separate the dominant directional errors and recover cancellation only after summation over the directional projection and correction residuals. Formally, we decompose the  projection error into leading directional projection errors and mixed-direction terms of order \(2k+2\), i.e.
	\begin{equation}
		I-\prod_{i=1}^d \bm P^{x_i}
		=
		\sum_{\emptyset\neq S\subset\{1,\dots,d\}}(-1)^{|S|+1}\prod_{i\in S}(I-\bm P^{x_i}).
	\end{equation}
	
	For clarity, we present the analysis in 2D. With projection \(\bm Q^*=\bm P^x\bm P^y\), we construct directional correction functions \(\bm X^l u\) and \(\bm\Gamma^l u\) in \Cref{def:correctfun2D} to remove the $x-$ and $y-$directional errors. The key new insight is the HOCA mechanism (\Cref{prop:correct_testfuncest2D}): an individual directional projection or correction residual need not have the desired superconvergent order, but the aggregate of the directional projection residual and all same-direction correction residuals does. This is the point at which the present CDG analysis departs from both the 1D CDG optimal-error theory and the tensor-product DG superconvergence theory in \cite{cao_shu_yang_zhang_2015,cao2018superconvergence,cao_zhang_zou_2014}.

\begin{theorem}[Corrected-initialization superconvergence of the 2D semidiscrete CDG method]\label[theorem]{thm:supercon2D}
	Consider the semidiscrete $\mathbb{Q}^k$ CDG scheme for \eqref{equ:conservation_law} on a 2D uniform Cartesian overlapping mesh. Let $\bm{Q}^\ast$ denote the projection from \Cref{def:projection2D}, and let $\{\bm{X}^l,\bm{\Gamma}^l\}$ denote the directional corrections from \Cref{def:correctfun2D}. Define
		\[
		\bm{\zeta}=\bm{u}_h - \bm{Q}^\ast u,
		\qquad
		\hat{\bm{\zeta}} := \bm{\zeta}+ \sum_{l=1}^{\min\{k,2\}} (\bm{X}^lu + \bm{\Gamma}^lu).
		\]
		If $u \in C^1([0,T]; H^{\sigma_k+1}(\Omega))$ and $\|\hat{\bm{\zeta}}(\cdot,0)\| \lesssim h^{\sigma_k}$, then
		\[
		\|\hat{\bm{\zeta}}(\cdot,T)\| \lesssim h^{\sigma_k}, \qquad \sigma_k := \min\{2k+1, k+3\}.
		\]
\end{theorem}
As in 1D, the proof starts from the error equation. The 2D setting requires additional work because the tensor-product projection lacks a simple one-cell variational
characterization analogous to the 1D LSZ projection. Although the projection error can be split into directional components, these components remain coupled through the transverse variable in the CDG residual and thus cannot be treated as independent 1D errors. The weak cancellation used in the 1D argument is therefore unavailable term by term. HOCA restores the required asymptotic weak cancellation only after aggregation. Apart from this new ingredient, the proof follows the same broad strategy, so we emphasize the parts specific to the 2D case.
\begin{description}
	\item[Step 1:] In the error equation, we decompose the projection error \(\bm{\eta}\) as
	\begin{equation}\label{equ:splitting}
		\bm{\eta} = \bm{\chi}u + \bm{\gamma}u - \tilde{\bm{\eta}},
	\end{equation}
	where
	\(
	\bm{\chi}u := (\bm I-\bm P^x)u,\quad
	\bm{\gamma}u := (\bm I-\bm P^y)u,\quad
	\tilde{\bm{\eta}} := (\bm I-\bm P^x)\bigl[(\bm I-\bm P^y)u\bigr].
	\)
	
	\item[Step 2:] Construct the directional correction functions \(\bm X^l u\) and \(\bm\Gamma^l u\) (see \cref{def:correctfun2D}) and define
	$\hat{\bm{\eta}}_x = \sum_{l=0}^{\min\{k,2\}}\bm X^l u$, $\hat{\bm{\eta}}_y = \sum_{l=0}^{\min\{k,2\}}\bm \Gamma^l u.$ By utilizing the Gauss--Lobatto projection, the initial terms $\bm{X}^0u$ and $\bm{\Gamma}^0u$ are designed to be continuous in the $y$- and $x$-directions, respectively.  
	
	\item[Step 3:] Establish error estimates for each correction term. For constant test functions in one coordinate direction, the leading-order residual terms are eliminated by the correction construction, as in the 1D analysis. The genuinely new point is the contribution from constant test functions, for which the weak cancellation is no longer available for each directional projection error separately and is recovered only after aggregation through the HOCA mechanism in \cref{prop:correct_testfuncest2D}.
\end{description}

The desired superconvergent estimate then follows from the above bounds together with stability and Gr\"onwall's inequality.

\begin{corollary}[Cell-average superconvergence in 2D]\label[corollary]{coro:supercon2D}
	Suppose the scheme is initialized with
	\(
	\bm{u}_h(\cdot,0)=\bm{Q}^\ast u(\cdot,0)
	-\bm{X}^1(u(\cdot,0))-\bm{\Gamma}^1(u(\cdot,0)).
	\)
	Under the assumptions of \Cref{thm:supercon2D}, the primal-mesh cell-average error satisfies
	\[
	e_{\mathrm{avg}}
	:= \left( h_xh_y\sum_{i,j}
	\left(\frac{1}{|C_{i,j}|}\int_{C_{i,j}} (u-u_h^C)\,dx\,dy\right)^2 \right)^{1/2} \lesssim h^{\sigma_k}.
	\]
\end{corollary}

\begin{proof}
The proof is analogous to \Cref{coro:supercon1D}. The tensor-product projection \(Q_C^\ast\) preserves primal-cell averages, and the directional corrections \(X_C^l u\) and \(\Gamma_C^l u\) have zero average in their construction direction; hence their cell averages over \(C_{i,j}\) vanish. Thus the cell-average error equals the cell average of the corrected error up to sign. Applying Cauchy's inequality cell by cell and then \Cref{thm:supercon2D} completes the proof.
\end{proof}

\subsection{Tensor-product projections and directional splitting}
\begin{definition}\label[definition]{def:projection2D}
	In 2D, we define the projection operators \( Q_C^* \) and \( Q_D^* \) as tensor products of 1D projections:
	\begin{align*}
		Q_C^* := P_{C}^x P_{C}^y, \qquad
		Q_D^* := P_{D}^x P_{D}^y.
	\end{align*}
	Here \( P_{C}^x, P_{D}^x \) act in the \( x \)-direction and \( P_C^y, P_D^y \) act in the \( y \)-direction on the primal/dual meshes (see Definition~\ref{def:projection}). 
	
	For a 2D function \(w_C\), and for each fixed \(y\), define the
	\(x\)-slice \(w_C^y(x):=w_C(x,y)\). We write
	\(
	\widehat P_i^x(w_C,\psi_i^x)
	:=
	\widehat P_i\bigl(w_C^y,\psi_i^x(\cdot,y)\bigr),
	\)
	where the operator on the right is the 1D form in
	\Cref{def:projection}. Similarly, for each fixed \(x\),
	\(
	\widehat P_j^y(w_C,\psi^y)
	:=
	\widehat P_j\bigl(w_C^x,\psi^y(x,\cdot)\bigr).
	\)
\end{definition}

This projection preserves \( \mathbb{Q}^k \) polynomials and satisfies the standard approximation property.

\begin{lemma}[Standard approximation]\label[lemma]{lem:standard_app2D}
	For the projection \( \bm{Q}^* \in \mathcal{L}(H^{k+1}(\Omega), L^2(\Omega)\times L^2(\Omega)) \), we have
	\begin{align*}
		\| (\bm{Q}^* - \bm{I})\omega \| + h \| (\bm{Q}^* - \bm{I})\omega \|_{\infty} + h^{1/2} \| (\bm{Q}^* - \bm{I})\omega \|_{\Gamma} \lesssim h^{k+1} \| \omega \|_{H^{k+1}}.
	\end{align*}
\end{lemma}

In addition to standard approximation properties, the projection \( \bm{Q}^* \) also exhibits a superconvergent estimate under the spatial operator analogous to the 1D case (see \cref{lem:supercon_project}).

\begin{lemma}[Superconvergence under spatial operator]\label[lemma]{lem:superconpro2D}
	Assume \(u\in H^{k+2}(\Omega)\). Then
	\begin{align*}
		\biggl| \sum_{i,j} \tilde{B}_{i+1/2, j+1/2}(\bm{\eta}, \phi_2) \biggr| \lesssim h^{k+1} \| u \|_{H^{k+2}} \| \phi_2 \|, \quad
		\biggl| \sum_{i,j} \hat{B}_{i,j}(\bm{\eta}, \phi_1) \biggr| \lesssim h^{k+1}  \| u \|_{H^{k+2}}\| \phi_1 \|.
	\end{align*}
\end{lemma}

\begin{proof}
	In 2D, we define the projection operators \(\bm{Q}^*\) as the tensor product of 1D projections $\bm{P}^x\bm{P}^y$. We decompose the spatial operator term as
	\begin{align}\label{equ:errorsplit}
		\hat{B}_{i,j}(\bm{\eta}, \phi_1) = \hat{B}_{i,j}(\bm{\chi}u, \phi_1) + \hat{B}_{i,j}(\bm{\gamma}u, \phi_1) - \hat{B}_{i,j}(\tilde{\bm{\eta}}, \phi_1).
	\end{align}
	We estimate the three terms in \eqref{equ:errorsplit} separately. Note that \( \bm{\gamma}u \) and \( \bm{\chi}u \) are continuous in the \( x \)- and \( y \)-directions, respectively.
	For any \( \bm{w}^y \) continuous in \( x \) and \( \bm{w}^x \) continuous in \( y \), integration by parts gives
	\begin{align}
		\hat{B}_{i,j}(\bm{w}^y, \phi_1) &= \int_{I_i} \hat{B}_j^y(\bm{w}^y, \phi_1) \,\mathrm{d}x - \int_{C_{i,j}} \partial_x w^y_{D} \phi_1 \,\mathrm{d}x\mathrm{d}y, \label{equ:Bwy} \\
		\hat{B}_{i,j}(\bm{w}^x, \phi_1) &= \int_{J_j} \hat{B}_i^x(\bm{w}^x, \phi_1) \,\mathrm{d}y - \int_{C_{i,j}} \partial_y w^x_{D} \phi_1 \,\mathrm{d}x\mathrm{d}y, \label{equ:Bwx}
	\end{align}
	where
	\begin{align*}
		\hat{B}_j^y(\bm{w}^y, \phi_1) &= \frac{1}{\tau_{\max}} \int_{J_j} (w_{D}^y - w_{C}^y) \phi_1 \,\mathrm{d}y \\
		&\quad + \int_{J_j} w_{D}^y \partial_y \phi_1 \,\mathrm{d}y - (w_{D}^y)_{j+1/2} (\phi_1)_{j+1/2}^- + (w_{D}^y)_{j-1/2} (\phi_1)_{j-1/2}^+, \\[1ex]
		\hat{B}_i^x(\bm{w}^x, \phi_1) &= \frac{1}{\tau_{\max}} \int_{I_i} (w_{D}^x - w_{C}^x) \phi_1 \,\mathrm{d}x \\
		&\quad + \int_{I_i} w_{D}^x \partial_x \phi_1 \,\mathrm{d}x - (w_{D}^x)_{i+1/2} (\phi_1)_{i+1/2}^- + (w_{D}^x)_{i-1/2} (\phi_1)_{i-1/2}^+.
	\end{align*}
	Applying the 1D superconvergent estimate under the spatial operator from \cref{lem:supercon_project} and the Cauchy--Schwarz inequality, we obtain
	\(
	\left| \sum_{i,j} \int_{I_i} \hat{B}_j^y(\bm{\gamma}u, \phi_1) \,\mathrm{d}x \right| \lesssim h^{k+1} \| u \|_{H^{k+2}} \| \phi_1 \|.
	\)
	Since \( \partial_x \bm{\gamma}u = \bm{\gamma}(\partial_x u) \) and \( \partial_y \bm{\chi}u = \bm{\chi}(\partial_y u) \), the remaining integrals can be estimated using standard projection approximation properties. Hence,
	\begin{align*}
		\biggl| \sum_{i,j} \hat{B}_{i,j}(\bm{\gamma}u, \phi_1) \biggr| &\lesssim h^{k+1} \| u \|_{H^{k+2}} \| \phi_1 \|,\quad
		\biggl| \sum_{i,j} \hat{B}_{i,j}(\bm{\chi}u, \phi_1) \biggr| \lesssim h^{k+1} \| u \|_{H^{k+2}} \| \phi_1 \|.
	\end{align*}
	For the last term in \eqref{equ:errorsplit}, we invoke the weak boundedness property from \cref{lem:weaklybdd}:
	\begin{equation}\label{equ:tensorerror}
		\biggl| \sum_{i,j} \hat{B}_{i,j}(\tilde{\bm{\eta}}, \phi_1) \biggr| \lesssim h^{-1} \| (\bm{I} -  \bm{P}^x)[(\bm{I} - \bm{P}^y) (u)] \| \| \phi_1 \| \lesssim h^{k+1} \| u \|_{H^{k+2}} \| \phi_1 \|.
	\end{equation}
	If $u$ is sufficiently smooth, we can improve the estimate in \eqref{equ:tensorerror} to $\mathcal{O}(h^{2k+1})$.
\end{proof}

\begin{remark}
	After splitting the projection error into directional components (in \( x \) and \( y \)), the subsequent analysis is not a simple lift of the 1D case. When inserted into the 2D spatial operator, additional terms on overlapping cells appear that distinguish the 2D analysis from 1D; for example,
	$
	\iint_{C_{i,j}} \partial_x w_{D,y} \phi_1 \,\mathrm{d}x\mathrm{d}y  
$ in \eqref{equ:Bwy}. 
	These terms must be accounted for when constructing the directional corrections in the next subsection.
\end{remark}

\subsection{Directional corrections and the HOCA mechanism}
As in 1D, correction functions are central to the superconvergent analysis. In 2D, the task is more complicated. First, the tensor-product projection lacks a variational formulation, making low-order error cancellation nontrivial. Moreover, basis-level corrections are more involved since \( \dim \mathbb{Q}^{k+1} - \dim \mathbb{Q}^{k} = 2k+3 \), i.e., \( \mathcal{O}(k) \) additional modes must be controlled, unlike the 1D case.

The error splitting in \eqref{equ:splitting} ({cf.~\cite{cao_shu_yang_zhang_2015}}) motivates constructing corrections by tensorizing the pre-correction operators  \( \bm{\chi}^l \) and \( \bm{\gamma}^l \) with the 1D Gauss--Lobatto projection in the \( x \)-direction (\( \bm{G}^x \)) or \( y \)-direction (\( \bm{G}^y \)).

\begin{definition}\label[definition]{def:correctfun2D}
	Using the 1D Gauss--Lobatto projection operator \( \bm{G} \) and the pre-corrections \( \bm{\chi}^lu \) and \( \bm{\gamma}^lu \), the corrections in $V_h\times W_h$ are defined as
	\[
	\bm{X}^{l}u := \bm{G}^{y}( \bm{\chi}^{l}(u)), \quad
	\bm{\Gamma}^{l}u := \bm{G}^{x} (\bm{\gamma}^{l}(u)),
	\]
	with \( \bm{\chi}^0u := \bm{\chi}u \) and \( \bm{\gamma}^0u := \bm{\gamma}u \). We set \(\bm{\chi}^{\,l}u=\bm{\gamma}^{\,l}u=0\) for \(l<0\).
	
	For $l>0$, the \( x \)-directional pre-correction \( \bm{\chi}^{l}u \in \hat{\mathbb{P}}_x^k(I_i) \times \hat{\mathbb{P}}_x^k(I_{i+1/2}) \) is defined recursively for any \( \bm{\psi} = (\psi_1, \psi_2) \in \hat{\mathbb{P}}_x^k(I_i) \times \hat{\mathbb{P}}_x^k(I_{i+1/2}) \) by:
	\begin{equation*}
		\resizebox{0.99\hsize}{!}{$
			\begin{aligned}
				\widehat{P}_i^{\,x}\bigl(-{\chi}_C^{l}u, \psi_1\bigr) &= \int_{I_i} (\bm{\chi}^{l-1}\partial_x u)\psi_1 \,\mathrm{d}x + \widehat{\mathcal{T}}_i^{\,x}\bigl(\psi_1; \bm{\chi}^{l-1}u\bigr) - \widehat{\mathcal{D}}_i^x(\bm{\chi}^{l-2}\partial_y u, \psi_1) \\
				&\quad + \int_{x_{i-1/2}}^{x_i} [\Delta_x^+( \chi_C^{l-1}\partial_y u)]\psi_1 \,\mathrm{d}x + \int_{x_i}^{x_{i+1/2}}[ \Delta_x^- (\chi_C^{l-1}\partial_y u)]\psi_1 \,\mathrm{d}x, \\
				\widetilde{P}_{i+1/2}^{\,x}\bigl(-{\chi}_D^{l}u, \psi_2\bigr) &= \int_{I_{i+1/2}}( \bm{\chi}^{l-1}\partial_x u)\psi_2 \,\mathrm{d}x + \widetilde{\mathcal{T}}_{i+1/2}^{\,x}\bigl(\psi_2; \bm{\chi}^{l-1}u\bigr) - \widetilde{\mathcal{D}}_{i+1/2}^x(\bm{\chi}^{l-2}\partial_y u, \psi_2) \\
				&\quad +\int_{x_i}^{x_{i+1/2}} [\Delta_x^+ (\chi_D^{l-1}\partial_y u)]\psi_2 \,\mathrm{d}x + \int_{x_{i+1/2}}^{x_{i+1}} [\Delta_x^- (\chi_D^{l-1}\partial_y u)]\psi_2 \,\mathrm{d}x.
			\end{aligned}$}
	\end{equation*}
	Here, \( \widehat{\mathcal{T}}_i^x(\cdot, \cdot) = \widehat{\mathcal{T}}_i(\cdot, \cdot) \) defined in \eqref{equ:defT}, and
	\begin{equation*}	\resizebox{0.99\hsize}{!}{$
			\begin{aligned}
				\widehat{\mathcal{D}}_i^x(\bm{w}, \psi_1) &:= \int_{I_i} w_D \psi_1 \,\mathrm{d}x - \int_{x_{i-1/2}}^{x_i} w_C(x+h_x/2)\psi_1 \,\mathrm{d}x - \int_{x_i}^{x_{i+1/2}} w_C(x-h_x/2)\psi_1 \,\mathrm{d}x, \\
				\widetilde{\mathcal{D}}_{i+1/2}^x(\bm{w}, \psi_2) &:= \int_{I_{i+\frac12}} w_C \psi_2 \,\mathrm{d}x - \int_{x_i}^{x_{i+1/2}} w_D(x+h_x/2)\psi_2 \,\mathrm{d}x - \int_{x_{i+1/2}}^{x_{i+1}} w_D(x-h_x/2)\psi_2 \,\mathrm{d}x.
			\end{aligned}$}\end{equation*}
	For the \( y \)-direction, \( \bm{\gamma}^lu \) is defined analogously by switching \( x \) to \( y \), \( \bm{\chi} \) to \( \bm{\gamma} \), \( I_i \) to \( J_j \), and \( I_{i+1/2} \) to \( J_{j+1/2} \).
\end{definition}

Notice that \( \bm{\chi}^{l}u \) and \( \bm{\gamma}^{l}u \) are defined via the 1D correction operator; thus, we can obtain estimates for them---and consequently for \( \bm{X}^{l}u \) and \( \bm{\Gamma}^{l}u \)---via the boundedness of the Gauss--Lobatto projection.

\begin{lemma}\label[lemma]{prop:correctestimate2D}
	The pre-correction functions satisfy the following bounds:
	\begin{equation}\label{equ:est_r}
		\| \bm{\chi}^{l}u \| \lesssim h^{k+1+l} \| u \|_{H^{k+1+l}}, \quad
		\| \bm{\gamma}^{l}u \| \lesssim h^{k+1+l} \| u \|_{H^{k+1+l}}.
	\end{equation}
	Furthermore, by tensoring with the 1D Gauss--Lobatto projection, the correction functions satisfy:
	\begin{equation}\label{equ:est_R}
		\| \bm{X}^{l}u \| \lesssim h^{k+1+l} \| u \|_{H^{k+1+l}}, \quad
		\| \bm{\Gamma}^{l}u \| \lesssim h^{k+1+l} \| u \|_{H^{k+1+l}}.
	\end{equation}
\end{lemma}

\begin{proof}
	To avoid repetition, we only discuss the estimate in the \( x \)-direction on the primal mesh. From \cref{prop:correctestimate} together with \cref{coro:overlap}, we obtain
	\[
	\|\chi_C^{l}u\|_{i,j} \leq h\Bigg( \|\chi_C^{l-1}\partial_x u \|_{i,j} + \frac{1}{\tau_{\max}} \|\chi_C^{l-1}(\Delta_x^\pm u)\|_{i,j} +  \|\chi_D^{l-2} (\Delta_x^\pm\partial_yu)\|_{i,j}+\|\Delta_x^\pm(\chi_C^{l-1} \partial_y u)\|_{i,j}\Bigg).
	\]
	Since \( \bm{\chi}^{0}u = (\bm{I} - \bm{P}^x)u \) and \( \bm{\chi}^{-1}u := \bm{0} \), the standard approximation theorem implies that \eqref{equ:est_r} holds for \( l=0 \). We can then show inductively that
	\[
	\| \chi_C^{l}u \|_{i,j} \lesssim h^{k+1+l} \| u \|_{H^{k+1+l}(C_{i,j})}.
	\]
	Using the boundedness of the Gauss--Lobatto projection, we get
	\[
	\| \bm{X}_C^{l}u \|_{i,j} \lesssim \| \chi_C^{l}u \|_{i,j} \lesssim h^{k+1+l} \| u \|_{H^{k+1+l}(C_{i,j})}.
	\]
\end{proof}

Based on \cref{prop:correctestimate2D}, the following lemma establishes the superconvergent property of the correction functions under the spatial operator. For example, on the primal mesh \( C_{i,j} \), we have:

\begin{lemma}\label[lemma]{lem:interface2D}
	Assume \( u \in H^{\sigma_k+1}(\Omega) \). Then for any \( (\psi_i^x, \psi_j^y) \in V_x^0 \times V_y^0 \subset \hat{\mathbb{P}}_x^k(I_i) \times \hat{\mathbb{P}}_y^k(J_j) \) and \(1\le l\le \min\{k,2\}\), we have
	\begin{align*}
		\biggl| \sum_{i,j} \left[ \hat{B}_{i,j}(\bm{\chi}^{l}u, \psi_i^x) - \int_{J_j} \widehat{P}_i^x(\chi_C^{l}u, \psi_i^x) \,\mathrm{d}y \right] \biggr| &\lesssim h^{k+1+l} \| u \|_{H^{k+2+l}} \| \psi_i^x \|, \\
		\biggl| \sum_{i,j} \left[ \hat{B}_{i,j}(\bm{\gamma}^{l}u, \psi_j^y) - \int_{I_i} \widehat{P}_j^y(\gamma_C^{l}u, \psi_j^y) \,\mathrm{d}x \right] \biggr| &\lesssim h^{k+1+l} \| u \|_{H^{k+2+l}} \| \psi_j^y \|.
	\end{align*}
\end{lemma}

\begin{proof}
	Since \( \bm{\chi}^lu \) is continuous in the \( y \)-direction, from \eqref{equ:Bwx}, \eqref{equ:Bwy}, and \cref{coro:overlap}, we have
	\[
	\hat{B}_{i,j}(\bm{\chi}^lu, \psi_i^x) - \int_{J_j} \widehat{P}_i^x(\chi_C^lu, \psi_i^x) \,\mathrm{d}y = \int_{J_j} \widehat{\mathcal{T}}_i^x(\psi_i^x; \bm{\chi}^lu) \,\mathrm{d}y - \int_{C_{i,j}} (\chi_D^l\partial_y u)\psi_i^x \,\mathrm{d}x\mathrm{d}y.
	\]
	Then \cref{coro:overlap} and \cref{prop:correctestimate2D} yield the desired estimate.
\end{proof}

\subsection{HOCA property}
In the 1D analysis, superconvergence follows from two
complementary mechanisms. The correction functions control the residual for
zero-mean test functions, while the asymptotic weak cancellation controls the
residual for constant test functions. In 2D, however, the weak
cancellation associated with an individual directional residual is not
sufficiently high order and would therefore degrade the superconvergence
rate. The optimal order is recovered only after the directional projection
residual is aggregated with the corresponding correction residuals.

We first state the resulting HOCA property. Its proof is organized around
three auxiliary lemmas. The first gives a local directional identity, which
decomposes the local spatial residual into the \(R_{i+\frac12}^l\)-terms
defined in \eqref{eq:def-R-local}, the
\(\mathcal I_{i+\frac12}^l\)-terms defined in \eqref{eq:def-I-local}, and the
remaining correction term. The second estimates the
\(\mathcal I_{i+\frac12}^l\)-terms in this local identity after aggregation
over the correction levels, while the third estimates the
\(R_{i+\frac12}^l\)-terms.
\begin{proposition}[HOCA property]\label[proposition]{prop:correct_testfuncest2D}
	For any $\bm{\psi} = (\psi_i^x,\psi_j^y)\in \overline{V}_x\times \overline{V}_y$, the following estimates hold
	\begin{align*}
		\sum_{i,j}	\biggl|  \sum_{l=0}^{\min\{k,2\}} \widehat{B}_{i,j}\bigl(\bm{X}^{l}u, \psi_i^x\bigr) \biggr| &\lesssim h^{\sigma_k} \|u\|_{H^{\sigma_k+1}(\Omega)}\,\|\psi_i^x\|,\\
		\sum_{i,j} \biggl|\sum_{l=0}^{\min\{k,2\}} \widehat{B}_{i,j}\bigl(\bm{\Gamma}^{l}u, \psi_j^y\bigr) \biggr| &\lesssim h^{\sigma_k} \|u\|_{H^{\sigma_k+1}(\Omega)}\,\|\psi_j^y\|.
	\end{align*}
	The same estimates hold on the dual mesh after replacing
	\(I_i,J_j,\widehat P,\widehat B\) by
	\(I_{i+\frac12},J_{j+\frac12},\widetilde P,\widetilde B\), respectively.
\end{proposition}

We use the following explicit identity for the constant-test terms.
\begin{lemma}[Local directional identity]
	\label[lemma]{lem:directional-local-identity}
	Let
	$
	v(x,y):=u(x,y)-u(x-h_x,y)
	$
	$
	\phi_{i+\frac12}^x:=\frac{2(x-x_{i+\frac12})}{h_x}
	$ on  $I_{i+\frac12}$. 
	For fixed \(y\), and for \(0\le l\le \min\{k,2\}\), define
	\begin{subequations}\label{eq:def-RI-local}
		\begin{align}
			R_{i+\frac12}^l(u;y)
			:={}&
			\int_{I_{i+\frac12}}
			\chi_D^{l-1}(\partial_xv)\phi_{i+\frac12}^x\,\mathrm dx
			+
			\widetilde{\mathcal T}_{i+\frac12}^{x}
			(\phi^x_{i+\frac12};\bm\chi^{l-1}v)
			-
			\widetilde{\mathcal D}_{i+\frac12}^{x}
			(\bm\chi^{l-2}\partial_yv;\phi^x_{i+\frac12}),
			\label{eq:def-R-local}
			\\
			\mathcal I_{i+\frac12}^l(u;y)
			:={}&
			\int_{I_{i+\frac12}}
			\chi_D^{l-1}\partial_xu\,\mathrm dx
			+
			\widetilde{\mathcal T}_{i+\frac12}^{x}
			(1;\bm\chi^{l-1}u)
			-
			\widetilde{\mathcal D}_{i+\frac12}^{x}
			(\bm\chi^{l-2}\partial_yu;1).
			\label{eq:def-I-local}
		\end{align}
	\end{subequations}
	Then, for any \(\psi_i^x\in\overline V_x\),
	\[
	\begin{aligned}
		\widehat B_{i,j}(\bm\chi^l u,\psi_i^x)
		=
		\int_{J_j}
		\left[
		\frac12R_{i+\frac12}^l(u;y)
		-
		\mathcal I_{i+\frac12}^l(u;y)
		+
		\int_{I_i}
		(\chi_D^{l-1}-\chi_D^l)\partial_yu\,\mathrm dx
		\right]\psi_i^x(y)\,\mathrm dy .
	\end{aligned}
	\]
\end{lemma}

\begin{proof}
	Since \(\bm\chi^l u\) is continuous in the \(y\)-direction, the 2D
	operator identity gives
	\[
	\widehat B_{i,j}(\bm\chi^l u,\psi_i^x)
	=
	\int_{J_j}
	\widehat B_i^x(\bm\chi^l u,\psi_i^x)\,\mathrm dy
	-
	\int_{J_j}\int_{I_i}
	\chi_D^l\partial_yu\,\mathrm dx\,\psi_i^x(y)\,\mathrm dy .
	\]
	Because \(\psi_i^x\in\overline V_x\), it is constant in \(x\) on each primal cell.
	Therefore,
	\[
	\widehat B_i^x(\bm\chi^l u,\psi_i^x)
	=
	\frac{\psi_i^x}{2}
	\left[
	\widetilde P_{i+\frac12}^{x}
	(-\chi_D^l u,\phi_{i+\frac12}^x)
	-
	\widetilde P_{i-\frac12}^{x}
	(-\chi_D^l u,\phi_{i-\frac12}^x)
	\right].
	\]
	
	By the directional recursive definition,
	\[
	\begin{aligned}
		\widetilde P_{i+\frac12}^{x}
		(-\chi_D^l u,\phi_{i+\frac12}^x)
		={}&
		\int_{I_{i+\frac12}}
		\chi_D^{l-1}\partial_xu\,
		\phi_{i+\frac12}^x\,\mathrm dx
		+
		\widetilde{\mathcal T}_{i+\frac12}^{x}
		(\phi_{i+\frac12}^x;\bm\chi^{l-1}u)\\&\quad -
		\widetilde{\mathcal{D}}_{i+\frac12}^{x}
		(\bm\chi^{l-2}\partial_yu;\phi_{i+\frac12}^x)
		+
		E_{i+\frac12}^l,
	\end{aligned}
	\]
	where
	$$
	E_{i+\frac12}^l
	:=
	\int_{x_i}^{x_{i+\frac12}}
	\Delta_x^+(\chi_D^{l-1}\partial_yu)
	\phi_{i+\frac12}^x\,\mathrm dx
	+
	\int_{x_{i+\frac12}}^{x_{i+1}}
	\Delta_x^-(\chi_D^{l-1}\partial_yu)
	\phi_{i+\frac12}^x\,\mathrm dx .
	$$
	Using
	$
	\phi_{i+\frac12}^x(x+h_x/2)=\phi_{i+\frac12}^x(x)+1$ 
	and 
	$\phi_{i+\frac12}^x(x-h_x/2)=\phi_{i+\frac12}^x(x)-1,
	$
	a direct half-cell change of variables yields
	$$
	E_{i+\frac12}^l
	=
	-\int_{x_i}^{x_{i+\frac12}}
	\chi_D^{l-1}\partial_yu\,\mathrm dx
	+
	\int_{x_{i+\frac12}}^{x_{i+1}}
	\chi_D^{l-1}\partial_yu\,\mathrm dx .
	$$ 
	The analogous expansion at \(i-\frac12\), followed by subtraction, gives
	\[
	\widehat B_i^x(\bm\chi^l u,\psi_i^x)
	=
	\left[
	\frac12\widetilde R_{i+\frac12}^l(u;y)
	-
	\mathcal I_{i+\frac12}^l(u;y)
	+
	\int_{I_i}\chi_D^{l-1}\partial_yu\,\mathrm dx
	\right]\psi_i^x(y),
	\]
	where \(\widetilde R_{i+\frac12}^l\) denotes the unshifted interface residual.
	Here we used the zero dual-cell average of
	\(\chi_D^{l-1}\partial_yu\) for \(l\ge1\), while \(l=0\) follows from the
	negative-index convention.
	
	Finally, shifting the terms over \(I_{i-\frac12}\) to \(I_{i+\frac12}\), using
	$
	\phi_{i-\frac12}^x(x-h_x)=\phi_{i+\frac12}^x(x) 
	$ 
	and the translation invariance of the local operators on the uniform periodic
	mesh gives
	$
	\widetilde R_{i+\frac12}^l(u;y)=R_{i+\frac12}^l(u;y).
	$
	Combining this identity with the transverse term 
	$
	-\int_{I_i}\chi_D^l\partial_yu\,\mathrm dx
	$ 
	gives the claimed local formula.
\end{proof}
 After aggregating the identity in each direction, we first estimate the aggregation of $\mathcal I_{i+\frac12}^l(u;y)$. 

\begin{lemma}[Estimate for the \(\mathcal{I}_{i+\frac12}^l\)-terms in the local identity]
	\label[lemma]{lem:constant-test-aggregation}
	Let \(m=\min\{k,2\}\). For \(\psi^x\in \overline V_x\), define 
	$
	\psi^x|_{C_{i,j}}=\psi_i^x(y).
	$ 
	Then
	\[
	\sum_j\sum_i
	\left|
	\int_{J_j}
	\left(
	\sum_{l=0}^{m}\mathcal I_{i+\frac12}^l(u;y)
	\right)
	\psi_i^x(y)\,\mathrm dy
	\right|
	\lesssim
	h^{\sigma_k}
	\|u\|_{H^{\sigma_k+1}(\Omega)}
	\|\psi^x\|.
	\]
\end{lemma}

\begin{proof}
	From the directional residual operator with constant test function \(1\),
	\[
	\begin{aligned}
		\widetilde B_{i+\frac12}^x(\bm\chi^{l-1}u;1)
		={}&
		\int_{I_{i+\frac12}}
		\chi_D^{l-1}\partial_xu\,\mathrm dx
		+
		\widetilde{\mathcal T}_{i+\frac12}^{x}
		(1;\bm\chi^{l-1}u)
		-
		\widetilde{\mathcal{D}}_{i+\frac12}^{x}
		(\bm\chi^{l-2}\partial_yu;1)
		+
		\int_{I_{i+\frac12}}
		\chi_C^{l-2}\partial_yu\,\mathrm dx .
	\end{aligned}
	\]
	Hence,  
	$
	\mathcal I_{i+\frac12}^l(u;y)
	=
	\widetilde B_{i+\frac12}^x(\bm\chi^{l-1}u;1)
	-
	\int_{I_{i+\frac12}}
	\chi_C^{l-2}\partial_yu\,\mathrm dx .
	$ 
	The shifted interface identity one correction level lower gives 
	$
	\mathcal I_{i+\frac12}^l(u;y)
	+
	\mathcal I_{i+\frac12}^{l-1}(u;y)
	=
	\mathcal A_{i+\frac12}^{l-1,x}(u;y),
	$ 
	where
	\[
	\begin{aligned}
		\mathcal A_{i+\frac12}^{l-1,x}(u;y)
		={}&
		\frac12
		\biggl[
		\int_{I_{i+\frac12}}
		\chi_D^{l-2}\partial_x\bigl(u(\cdot-h_x,y)-u(\cdot,y)\bigr)\phi^x\,\mathrm dx
		\\ & +
		\widetilde{\mathcal T}_{i+\frac12}^{x}
		\bigl(\phi^x;\bm\chi^{l-2}(u(\cdot-h_x,y)-u(\cdot,y))\bigr)
		-
		\widetilde{\mathcal{D}}_{i+\frac12}^{x}
		\bigl(\bm\chi^{l-3}\partial_y(u(\cdot-h_x,y)-u(\cdot,y));\phi^x\bigr)
		\biggr].
	\end{aligned}
	\]
	Since all negative correction levels vanish, 
	$
	\mathcal I_{i+\frac12}^0=0.
	$
	For \(l=1\), the 1D weak cancellation for $\bm{\chi}^0u$, together with the convention
	\(\bm\chi^{-1}=0\), gives
	$
	\mathcal I_{i+\frac12}^1=0.
	$ 
	Therefore
	\[
	\sum_{l=0}^{m}\mathcal I_{i+\frac12}^l
	=
	\begin{cases}
		0, & m=0~\mbox{or}~1,\\[1mm]
		\mathcal A_{i+\frac12}^{1,x}(u;y), & m=2.
	\end{cases}
	\]
	
	It remains only to estimate the case \(m=2\). The remaining term is precisely
	the first-level 1D correction residual with constant test function,
	namely
	$
	\mathcal A_{i+\frac12}^{1,x}(u;y)
	=
	\widehat B_i^x(R^1u,1),
	$ 
	where \(R^1u\) denotes the 1D first-level correction function.
	Thus, by the 1D asymptotic weak cancellation estimate
	\Cref{prop:correct_testfuncest}, applied with the piecewise constant test function \(\psi^x\),
	\[
	\begin{aligned}
		\sum_j\sum_i
		\left|
		\int_{J_j}
		\left(
		\sum_{l=0}^{m}\mathcal I_{i+\frac12}^l(u;y)
		\right)
		\psi_i^x(y)\,\mathrm dy
		\right|&=
		\sum_j\sum_i
		\left|
		\int_{J_j}
		\widehat B_i^x(R^1u,1)\psi_i^x(y)\,\mathrm dy
		\right| \\&\quad\lesssim
		h^{\sigma_k}
		\|u\|_{H^{\sigma_k+1}(\Omega)}
		\|\psi^x\|.
	\end{aligned}
	\]
	This proves the result.
\end{proof}

The aggregation of $R_{i+\frac12}^l(u;y)$ is bounded by the following lemma.
\begin{lemma}[Estimate for the \(R_{i+\frac12}^l\)-terms in the local identity]
	\label[lemma]{lem:grouped-RI-estimate}
	Let \(m=\min\{k,2\}\). For \(\psi^x\in \overline V_x\), define 
	$
	\psi^x|_{C_{i,j}}=\psi_i^x(y).
	$ 
	Then
	\[
	\sum_j\sum_i
	\left|
	\int_{J_j}
	\frac12\sum_{l=0}^{m}R_{i+\frac12}^l(u;y)
	\psi_i^x(y)\,\mathrm dy
	\right|
	\lesssim
	h^{\sigma_k}
	\|u\|_{H^{\sigma_k+1}(\Omega)}
	\|\psi^x\|.
	\]
\end{lemma}

\begin{proof}
	Set
	$
	A_{i+\frac12}(y)
	=
	\frac12\sum_{l=0}^{m}R_{i+\frac12}^l(u;y)
	-
	\sum_{l=0}^{m}\mathcal I_{i+\frac12}^l(u;y),
	$ 
	and define
	$
	v(x,y)=u(x,y)-u(x-h_x,y) $ on $ 
	\omega_{i+\frac12}^x=I_i\cup I_{i+1}.
	$
	We estimate the three possible values of \(m\). The case \(m=0\) is immediate
	from the definitions and the vanishing convention for negative correction
	levels.
	
	When \(m=1\), the constant-test aggregation lemma gives
	$
	\sum_{l=0}^{1}\mathcal I_{i+\frac12}^l(u;y)=0.
	$ 
	Moreover, the \(l=1\) directional residual block is the 1D
	first-level correction residual with constant test function,
	$
	\frac12 R_{i+\frac12}^{1}(u;y)
	=
	\widehat B_i^x(R^1u,1).
	$ 
	Thus, by \Cref{prop:correct_testfuncest}, applied with the locally constant
	test value
	\(\operatorname{sgn}(R_{i+\frac12}^{1})\psi_i^x(y)\), we obtain
	\[
	\sum_j\sum_i
	\left|
	\int_{J_j}
	\frac12 R_{i+\frac12}^{1}(u;y)\psi_i^x(y)\,\mathrm dy
	\right|
	\lesssim
	h^{\sigma_k}
	\|u\|_{H^{\sigma_k+1}(\Omega)}
	\|\psi^x\|.
	\]
	This proves the desired estimate for \(m=1\).
	
	It remains to consider \(m=2\). The \(l=1\) part is treated exactly as above,
	namely
	$
	\frac12 R_{i+\frac12}^{1}(u;y)=\widehat B_i^x(R^1u,1),
	$
	and hence its tested contribution is bounded by
	$
	h^{\sigma_k}
	\|u\|_{H^{\sigma_k+1}(\Omega)}
	\|\psi^x\|.
	$ 
	For the \(l=2\) part, the correction-function bounds, the trace estimate, and
	the finite-difference structure induced by the shifting property give
	\[
	\begin{aligned}
		|R_{i+\frac12}^2(u;y)|
		&\lesssim
		h^{1/2}
		\bigl(
		\|\chi_D^{1}\partial_x v\|_{I_{i+\frac12}}
		+h^{-1}\|\bm\chi^{1}(\Delta_x^\pm v)\|_{I_{i+\frac12}}
		+\|\bm\chi^{0}(\Delta_x^\pm\partial_y v)\|_{I_{i+\frac12}}
		\bigr)
		\\&\quad \lesssim
		h^{\sigma_k+\frac12}
		\Bigl(
		\|u(\cdot,y)\|_{H^{\sigma_k+1}(\omega_{i+\frac12}^x)}
		\Bigr),
	\end{aligned}
	\]
	where \(\omega_{i+\frac12}^x=I_i\cup I_{i+1}\). Multiplying by
	\(\psi_i^x(y)\), summing over \(i,j\), and using Cauchy--Schwarz together with
	the uniformly bounded overlap of \(\omega_{i+\frac12}^x\), we obtain directly
	\[
	\sum_j\sum_i
	\left|
	\int_{J_j}
	R_{i+\frac12}^{2}(u;y)\psi_i^x(y)\,\mathrm dy
	\right|
	\lesssim
	h^{\sigma_k}
	\|u\|_{H^{\sigma_k+1}(\Omega)}
	\|\psi^x\|.
	\]
	The \(l=1\) part has already been written as
	$
	\frac12 R_{i+\frac12}^{1}(u;y)
	=
	\widehat B_i^x(R^1u,1),
	$ 
	and is controlled by the 1D asymptotic weak cancellation estimate.
This yields
	$$
	\sum_j\sum_i
	\left|
	\int_{J_j}
	\frac12\sum_{l=0}^{m}R_{i+\frac12}^l(u;y)\psi_i^x(y)\,\mathrm dy
	\right|
	\lesssim
	h^{\sigma_k}
	\|u\|_{H^{\sigma_k+1}(\Omega)}
	\|\psi^x\|.
	$$
\end{proof}

We now prove \Cref{prop:correct_testfuncest2D}.

\begin{proof}
	We prove the \(x\)-direction estimate. The \(y\)-direction estimate follows
	by interchanging \(x\) and \(y\).
	All correction functions with negative superscripts are understood to be zero. Set \(m:=\min\{k,2\}\). For each \(0\le l\le m\), we split
	\[
	\widehat B_{i,j}(\bm X^l u,\psi_i^x)
	=
	\widehat B_{i,j}\bigl((\bm G^y-\bm I)\bm\chi^l u,\psi_i^x\bigr)
	+
	\widehat B_{i,j}(\bm\chi^l u,\psi_i^x).
	\]	
		The local weak boundedness of \(\widehat B_{i,j}\) gives
		\[
		\sum_{i,j}
		\left|
		\widehat B_{i,j}
		\bigl((\bm G^y-\bm I)\bm\chi^l u,\psi_i^x\bigr)
		\right|
		\lesssim
		h^{-1}
		\|(\bm G^y-\bm I)\bm\chi^l u\|
		\|\psi_i^x\|.
		\]
		Because \(\bm\chi^l\) is an \(x\)-directional correction, it commutes with
		\(\partial_y\). Hence, using the standard approximation estimate for \(\bm G^y\) and \Cref{prop:correctestimate2D}, we have
		\begin{align}\label{equ:transverse--GLest}
			\sum_{i,j}\left|
			\widehat B_{i,j}\bigl((\bm G^y-\bm I)\bm\chi^l u,\psi_i^x\bigr)
			\right|\lesssim 	h^{-1}\|(\bm G^y-\bm I)\bm\chi^l u\|\,\|\psi_i^x\|
			\lesssim
			h^{\sigma_k}
			\|u\|_{H^{\sigma_k+1}}\,\|\psi_i^x\|.
		\end{align}
	It remains to estimate the contribution of \(\bm\chi^l u\). For this purpose, we set
	\[
	\begin{aligned}
		R_{i+\frac12}^l(u;y)
		:={}&
		\int_{I_{i+\frac12}}
		\chi_D^{l-1}(\partial_xv_i)\phi^x\,\mathrm dx
		+
		\widetilde{\mathcal T}_{i+\frac12}^{x}
		(\phi^x;\bm\chi^{l-1}v_i)
		-
		\widetilde{\mathcal{D}}_{i+\frac12}^{x}
		(\bm\chi^{l-2}\partial_yv_i;\phi^x),
	\end{aligned}
	\]
	and
	$$
	\begin{aligned}
		\mathcal I_{i+\frac12}^l(u;y)
		:={}&
		\int_{I_{i+\frac12}}
		\chi_D^{l-1}\partial_xu\,\mathrm dx
		+
		\widetilde{\mathcal T}_{i+\frac12}^{x}
		(1;\bm\chi^{l-1}u)
		-
		\widetilde{\mathcal{D}}_{i+\frac12}^{x}
		(\bm\chi^{l-2}\partial_yu;1),
	\end{aligned}
	$$ for fixed $y$ with 
	\(
	v_i(x,y):=u(x,y)-u(x-h_x,y)\) and \(
	\phi^x:=\frac{2(x-x_{i+\frac12})}{h_x}
	\quad\text{on } I_{i+\frac12}.\)
	The local calculation in \cref{lem:directional-local-identity} gives
	\[
	\begin{aligned}
		\widehat B_{i,j}(\bm\chi^l u,\psi_i^x)
		=
		\int_{J_j}
		\left[
		\frac12R_{i+\frac12}^l(u;y)
		-
		\mathcal I_{i+\frac12}^l(u;y)
		+
		\int_{I_i}
		(\chi_D^{l-1}-\chi_D^l)\partial_yu\,\mathrm dx
		\right]\psi_i^x(y)\,\mathrm dy .
	\end{aligned}
	\]
	Summing over \(l=0,\ldots,\min\{k,2\}\), and using \(\chi_D^{-1}=0\), yields
	\[
	\begin{aligned}
		\sum_{l=0}^{\min\{k,2\}}
		\widehat B_{i,j}(\bm\chi^l u,\psi_i^x)
		={}&
		\int_{J_j}
		\left[
		\frac12\sum_{l=0}^{\min\{k,2\}}R_{i+\frac12}^l(u;y)
		-
		\sum_{l=0}^{\min\{k,2\}}\mathcal I_{i+\frac12}^l(u;y)
		\right]\psi_i^x(y)\,\mathrm dy
		\\
		&-
		\int_{J_j}\int_{I_i}
		\chi_D^{\min\{k,2\}}\partial_yu\,\mathrm dx\,\psi_i^x(y)\,\mathrm dy .
	\end{aligned}
	\]
	
	By \Cref{lem:grouped-RI-estimate,lem:constant-test-aggregation}, we have
	\[
	\begin{aligned}
		&\sum_{i,j}
		\left|
		\int_{J_j}
		\left[
		\frac12\sum_{l=0}^{\min\{k,2\}}R_{i+\frac12}^l(u;y)
		-
		\sum_{l=0}^{\min\{k,2\}}\mathcal I_{i+\frac12}^l(u;y)
		\right]\psi_i^x(y)\,\mathrm dy
		\right|\lesssim
		h^{\sigma_k}
		\|u\|_{H^{\sigma_k+1}(\Omega)}
		\|\psi_i^x\|.
	\end{aligned}
	\]
	Here the constant-test terms \(\mathcal I_{i+\frac12}^l\) are aggregated before
	estimation; the cancellation mechanism is proved in
	\Cref{lem:constant-test-aggregation}. In particular,
	\(
	\sum_{l=0}^{m}\mathcal I_{i+\frac12}^l
	=
	0
	\quad\text{for }m=0,1,
	\)
	while for \(m=2\) the sum reduces to a first-level 1D asymptotic
	weak cancellation property. It remains to bound the last transverse term. By the correction estimate,
	\[
	\begin{aligned}
		&\sum_{i,j}
		\left|
		\int_{J_j}\int_{I_i}
		\chi_D^{\min\{k,2\}}\partial_yu\,\mathrm dx\,\psi_i^x(y)\,\mathrm dy
		\right|
		\le
		\|\bm\chi^{\min\{k,2\}}(\partial_yu)\|\,\|\psi_i^x\|
		\lesssim
		h^{\sigma_k}
		\|u\|_{H^{\sigma_k+1}(\Omega)}
		\|\psi_i^x\|.
	\end{aligned}
	\]
	Combining the preceding estimates with \eqref{equ:transverse--GLest}, this
	implies
	\[
	\sum_{i,j}	\left|
	\sum_{l=0}^{m}
	\widehat B_{i,j}(\bm X^l u,\psi_i^x)
	\right|
	\lesssim
	h^{\sigma_k}
	\|u\|_{H^{\sigma_k+1}(\Omega)}
	\|\psi_i^x\|.
	\]
	This proves the \(x\)-direction estimate.
\end{proof}

\subsection{Proof of \cref{thm:supercon2D}}
From \eqref{equ:splitting}, we decompose the residual of the corrected error equation on the primal cell \( C_{i,j} \) into three components
\[
\biggl| \int_{C_{i,j}} \partial_t \hat{\zeta}_C \phi_1 \,\mathrm{d}x\mathrm{d}y - \hat{B}_{i,j}(\hat{\bm{\zeta}}, \phi_1) \biggr| \leq |\mathcal{C}_{i,j}^x(\phi_1)| + |\mathcal{C}_{i,j}^y(\phi_1)| + |\mathcal{C}_{i,j}^{xy}(\phi_1)|,
\]
where
\begin{align}
	\mathcal{C}_{i,j}^x(\phi_1) &:= \int_{C_{i,j}}  \sum_{l=0}^{\min\{k,2\}}\partial_t (	\bm{X}_C^{l}u) \phi_1 \,\mathrm{d}x\mathrm{d}y - \hat{B}_{i,j}(\hat{\bm{\eta}}_x, \phi_1),\label{equ:Cijx} \\
	\mathcal{C}_{i,j}^y(\phi_1) &:= \int_{C_{i,j}} \sum_{l=0}^{\min\{k,2\}}\partial_t( \bm{\Gamma}_C^{l}u) \phi_1 \,\mathrm{d}x\mathrm{d}y - \hat{B}_{i,j}(\hat{\bm{\eta}}_y, \phi_1),\label{equ:Cijy} \\
	\mathcal{C}_{i,j}^{xy}(\phi_1) &:= \int_{C_{i,j}} \partial_t \tilde{\eta}_C \phi_1 \,\mathrm{d}x\mathrm{d}y - \hat{B}_{i,j}(\tilde{\bm{\eta}}, \phi_1).
\end{align}

\subsubsection*{Estimate of \( |\mathcal{C}_{i,j}^{xy}(\phi_1)| \):}
From the local weak boundedness (see \cref{lem:weaklybdd}), we directly obtain 
$
\sum_{i,j} |\mathcal{C}_{i,j}^{xy}(\phi_1)| \lesssim h^{\sigma_k} \| \phi_1 \|.
$ 
The estimates in the \( x \)- and \( y \)-directions are analogous;  therefore, we present the proof for one representative case.

\subsubsection*{Estimate of \( |\mathcal{C}_{i,j}^y(\phi_1)| \):}
We split the term \( \mathcal{C}_{i,j}^y(\phi_1) \) into two parts:
\begin{equation}\label{equ:Cysplit}
	|\mathcal{C}_{i,j}^y(\phi_1)| \leq |\mathcal{C}_{i,j}^y(\phi^0)| + |\mathcal{C}_{i,j}^y(\phi)|,
\end{equation}
where \( \phi_1 \in V_h \) is split into \( \phi^0 \in V_y^0 \) and \( \phi \in \overline{V}_y \). 
\Cref{prop:correct_testfuncest2D} estimates the second term in \eqref{equ:Cysplit}:
$
\sum_{i,j} |\mathcal{C}_{i,j}^y(\phi)| \lesssim h^{\sigma_k} \|\phi\|.
$ 
From \Cref{def:correctfun2D},
\[
\begin{aligned}
	\mathcal C_{i,j}^y(\phi^0)
	={}&
	\sum_{l=0}^{\min\{k,2\}}
	\int_{I_i}
	\biggl[
	\widehat P_j^y\bigl(\gamma_C^l u,\phi^0\bigr)
	+
	\bigl(\gamma_C^{l-1}\partial_yu,\phi^0\bigr)_{J_j}
	+
	\widehat{\mathcal T}_j^y
	\bigl(\phi^0;\bm\gamma^{l-1}u\bigr)
	\\
	&\qquad\qquad
	+
	\int_{J_j}
	(\gamma_C^{l-1}\partial_xu-\gamma_D^{l-1}\partial_xu)\,\phi^0\,\mathrm dy
	\biggr]
	\,\mathrm dx
	\\
	&+
	\biggl(
	\partial_t(I-G_C^x)
	\sum_{l=0}^{\min\{k,2\}}\gamma_C^l u,\phi^0
	\biggr)_{i,j}
	+
	\widehat B_{i,j}
	\biggl(
	(I-\bm G^x)
	\sum_{l=0}^{\min\{k,2\}}\bm\gamma^l u,\phi^0
	\biggr).
\end{aligned}
\]
\Cref{lem:interface2D} combined with \Cref{prop:correctestimate2D} yields the desired estimate for the first term in \eqref{equ:Cysplit}: 
$
\sum_{i,j} |\mathcal{C}_{i,j}^y(\phi^0)| \lesssim h^{\sigma_k} \|\phi^0\|.
$

Then, summing over all cells yields
\[
\sum_{i,j} \bigl( |\mathcal{C}_{i,j}^x(\phi_1)| + |\mathcal{C}_{i,j}^y(\phi_1)| + |\mathcal{C}_{i,j}^{xy}(\phi_1)| \bigr) \lesssim h^{\sigma_k} \|\phi_1\|.
\]
The same bound also holds on the dual mesh. Taking \( \bm{\phi} = \hat{\bm{\zeta}} \) and summing the estimates over all primal and dual cells, we obtain the \( L^2 \) estimate: 
$
\frac{1}{2} \frac{d}{dt} \| \hat{\bm{\zeta}} \|^2 \lesssim h^{\sigma_k} \| \hat{\bm{\zeta}} \|.
$
Assuming \( \|\hat{\bm{\zeta}}(\cdot, 0) \|\lesssim  h^{\sigma_k} \), an application of Gr\"onwall's inequality completes the proof of the superconvergence theorem.

\subsection{Proof of the LSZ projection superconvergence rate}
	\label{subsec:baseLSZ}
	
	The preceding analysis identifies the asymptotic weak cancellation responsible for the superconvergence of the CDG method. We now prove the pointwise superconvergent phenomenon under the LSZ projection initialization in 1D and tensor-product projection initialization in 2D. The key point is that the discrete \(\ell_h^2\) pointwise estimate, together with the corresponding estimates for difference quotient operators, can be upgraded to an \(\ell^\infty\) estimate by a discrete Sobolev inequality on the sampled superconvergent points. For later use, define the uncorrected projection operator \(\bm{P}:=\bm{P}^*\) in 1D and \(\bm{P}:=\bm{Q}^*\) in 2D.

	Let \(G_K\) denote the local set of superconvergent points in a primal cell \(K\). In 1D, \(G_K\) is the set of roots in \(I_i\) of \((x-x_i)^{k+1}-P_C^\ast (x-x_i)^{k+1}\); in 2D, it is the tensor product of the corresponding 1D root sets in the \(x\)- and \(y\)-directions. We write \(G\) for the global set of all such points and \(G^{(m)}\) for the collection of points with a fixed local label \(m\). For a function \(w\) sampled on \(G\), define
	\[
	\|w\|_{\ell_h^2(G)}:=\left(\sum_{K\in\Gamma_h^C}|K|\max_{\xi\in G_K}|w(\xi)|^2\right)^{1/2},
	\qquad
	\|w\|_{\ell^\infty(G)}:=\max_{K\in\Gamma_h^C}\max_{\xi\in G_K}|w(\xi)|.
	\]
	The corresponding norms on \(G^{(m)}\) are defined analogously, without the maximum over local labels.
	
For \(\bm\alpha\in\{0,1\}^d\), let \(D^{\bm\alpha}=(D_x^+)^{\alpha_1}\) in 1D with \(D_x^+w=(w(\cdot+h)-w)/h\), and let \(D^{\bm\alpha}=(D_x^+)^{\alpha_1}(D_y^+)^{\alpha_2}\) in 2D with \(D_x^+w=(w(\cdot+h_x,\cdot)-w)/h_x\) and \(D_y^+w=(w(\cdot,\cdot+h_y)-w)/h_y\).
	
	\begin{lemma}[$\ell_h^2$ pointwise estimate based on the LSZ projection]
		\label[lemma]{lem:l2point}
		Assume the regularity required by the corrected residual estimates in \Cref{thm:supercon1D,thm:supercon2D} and \(u\in L^\infty([0,T];W^{k+2,\infty}(\Omega))\). If the semidiscrete CDG solution is initialized by the corresponding uncorrected projection \(\bm P\), then
		\(
		\|u(\cdot,T)-u_h^C(\cdot,T)\|_{\ell_h^2(G)}\lesssim h^{k+2}.
		\)
	\end{lemma}
	
	\begin{proof}
	We first consider the 1D case. Let \(\widehat u\) be the
	degree-\((k+1)\) Taylor polynomial of \(u(\cdot,T)\) at the center of
	\(I_i\). For \(x\in G_i\), the triangle inequality and the
	\(L^\infty\)-stability of \(P_C^\ast\) give
	\[
	|u(x,T)-u_h^C(x,T)|
	\lesssim
	\|u-\widehat u\|_{L^\infty(I_i)}
	+
	|\widehat u(x)-P_C^\ast\widehat u(x)|
	+
	|\zeta_C(x,T)|.
	\]
	The Taylor remainder is \(\mathcal O(h^{k+2})\). Moreover, since
	\(P_C^\ast\) is exact on \(\mathbb P^k\), the projection error of
	\(\widehat u\) is determined only by its degree-\((k+1)\) part, which
	vanishes at the superconvergent points \(G_i\). Hence
	\[
	\|u(\cdot,T)-u_h^C(\cdot,T)\|_{\ell_h^2(G)}
	\lesssim
	h^{k+2}
	+
	\|\zeta_C(\cdot,T)\|_{\ell_h^2(G)} .
	\]
	
	In 2D, at the tensor-product superconvergent points,
	\(u-u_h^C=(u-Q_C^\ast u)-\zeta_C\). For the projection part, using
	\(I-Q_C^\ast=(I-P_C^{x,\ast})+P_C^{x,\ast}(I-P_C^{y,\ast})\),
	the 1D root estimate in each coordinate direction and the
	\(L^\infty\)-stability of \(P_C^{x,\ast}\) yield
	$
	\|u(\cdot,T)-Q_C^\ast u(\cdot,T)\|_{\ell_h^2(G)}
	\lesssim h^{k+2}.
	$ 
	Therefore, in both 1D and 2D, it remains only to estimate
	\(\zeta_C\) in the corresponding discrete \(\ell_h^2\) norm.
	
	Let
	\[
	\mathcal C_{h}u
	:=
	\begin{cases}
		\displaystyle\sum_{l=1}^{m}R_C^l u, & d=1,\\[1ex]
		\displaystyle\sum_{l=1}^{m}(X_C^l u+\Gamma_C^l u), & d=2,
	\end{cases}
	\qquad m=\min\{k,2\}.
	\]
	Then
	$
	\zeta_C=\hat\zeta_C-\mathcal C_{h}u.
	$ 
	Repeating the corrected-error superconvergence estimate in
	\Cref{thm:supercon1D,thm:supercon2D} with the weaker initial bound
	\(\|\hat{\bm\zeta}(0)\|\lesssim h^{k+2}\) gives
	\(
	\|\hat{\bm\zeta}(T)\|\lesssim h^{k+2},
	\)
	because the residual forcing is of order \(h^{\sigma_k}\) and
	\(\sigma_k\ge k+2\).
	
	Consequently, by the inverse inequality on the finite set of superconvergent
	points and the correction estimates,
	\[
	\begin{aligned}
		\|\zeta_C(\cdot,T)\|_{\ell_h^2(G)}
		&\le
		\|\hat\zeta_C(\cdot,T)\|_{\ell_h^2(G)}
		+
		\|\mathcal C_hu(\cdot,T)\|_{\ell_h^2(G)}
		\lesssim
		\|\hat{\bm\zeta}(T)\|
		+
		\|\mathcal C_hu(\cdot,T)\|
		\lesssim h^{k+2}.
	\end{aligned}
	\]
	Combining this with the projection-root estimate above proves the lemma.
\end{proof}
	
Since the equation is linear and the mesh is uniform, the difference quotient
operators commute with the CDG spatial operator, the projection, and the
correction operators. In particular,
\(
D^{\bm\alpha}\bm u_h(\cdot,0)
=
D^{\bm\alpha}\bm P u(\cdot,0)
=
\bm P D^{\bm\alpha}u(\cdot,0).
\)
Hence \(D^{\bm\alpha}\bm u_h\) is the CDG solution for
\(D^{\bm\alpha}u\) with the LSZ projection initialization. Applying
\Cref{lem:l2point} to \(D^{\bm\alpha}u\) gives the following estimate.

	\begin{corollary}[Difference-quotient pointwise estimate]
		\label[corollary]{lem:l2diffpoint}
		Let \(\bm\alpha\in\{0,1\}^d\). Assume
		$$
		u\in C^1([0,T];H^{\sigma_k+1+|\bm\alpha|}(\Omega))
		\cap
		L^\infty([0,T];W^{k+2+|\bm\alpha|,\infty}(\Omega)).
		$$ 
		Then, under the LSZ projection initialization, 
		$$
		\|D^{\bm\alpha}(u(\cdot,T)-u_h^C(\cdot,T))\|_{\ell_h^2(G)}\lesssim h^{k+2}.
		$$
	\end{corollary}

\begin{lemma}[Discrete Sobolev inequality at superconvergent points]
	\label[lemma]{lem:discrete_sobolev_roots}
	For \(d=1,2\) and any periodic function \(v\),
	$$
	\|v\|_{\ell^\infty(G)}
	\lesssim
	\sum_{\bm\alpha\in\{0,1\}^d}\|D^{\bm\alpha}v\|_{\ell_h^2(G)}.
	$$
\end{lemma}

\begin{proof}
	It is enough to prove the estimate for a fixed local label \(m\), since the number of labels depends only on \(k\).
	
	In 1D, set \(a_i=v(x_i^m)\). The discrete average gives an index \(i_m\) such that \(|a_{i_m}|\lesssim \|a\|_{\ell_h^2}\). For any \(i\), periodic telescoping and Cauchy's inequality yield
	\[
	|a_i|\le |a_{i_m}|+\sum_q h |D_x^+a_q|
	\lesssim \|a\|_{\ell_h^2}+\|D_x^+a\|_{\ell_h^2}.
	\]
	Taking the maximum over the at most $k+1$ local labels gives the 1D estimate.
	
	In 2D, let \((p,q)\) enumerate the tensor-product root labels. We write
	\(
	a_{ij}^{(p,q)}=v(x_i^{p},y_j^{q})
	\)
	and abbreviate it as \(a_{ij}\) in the proof. For fixed \(j\), the 1D estimate in the \(x\)-direction gives \(\max_i|a_{ij}|\lesssim A_j+B_j\), where
	\[
	A_j=\left(h_x\sum_i |a_{ij}|^2\right)^{1/2},
	\qquad
	B_j=\left(h_x\sum_i |D_x^+a_{ij}|^2\right)^{1/2}.
	\]
	Applying the 1D estimate in the \(y\)-direction to \(A_j\) and \(B_j\), together with the reverse triangle inequality, gives
	\[
	\max_j A_j\lesssim \|a\|_{\ell_h^2}+\|D_y^+a\|_{\ell_h^2},
	\qquad
	\max_j B_j\lesssim \|D_x^+a\|_{\ell_h^2}+\|D_x^+D_y^+a\|_{\ell_h^2}.
	\]
	Therefore,
	\[
	\max_{i,j}|a_{ij}|
	\lesssim
	\|a\|_{\ell_h^2}
	+\|D_x^+a\|_{\ell_h^2}
	+\|D_y^+a\|_{\ell_h^2}
	+\|D_x^+D_y^+a\|_{\ell_h^2}.
	\]
	Taking the maximum over the finitely many local labels proves the result.
\end{proof}
	
	We now obtain the LSZ-projection-based \(\ell^\infty\) pointwise superconvergence theorem.
	
	\begin{theorem}[LSZ-projection-based \(\ell^\infty\) pointwise superconvergence]
		\label[theorem]{thm:base_LSZ_maximum}
		Let \(d\in\{1,2\}\) and assume
		$$
		u\in C^1([0,T];H^{\sigma_k+1+d}(\Omega))
		\cap L^\infty([0,T];W^{k+2+d,\infty}(\Omega)) $$ with 
		$\sigma_k=\min\{2k+1,k+3\}.$ 
		If the semidiscrete CDG method is initialized by the uncorrected projection \(\bm u_h(\cdot,0)=\bm{P}u(\cdot,0)\), then
		\[
		\max_{K\in\Gamma_h^C}\max_{\xi\in G_K}|u(\xi,T)-u_h^C(\xi,T)|\lesssim h^{k+2}.
		\]
	\end{theorem}
	
	\begin{proof}
		Take \(v=u(\cdot,T)-u_h^C(\cdot,T)\) in \Cref{lem:discrete_sobolev_roots}. By \Cref{lem:l2diffpoint}, \(\|D^{\bm\alpha}v\|_{\ell_h^2(G)}\lesssim h^{k+2}\) for all \(\bm\alpha\in\{0,1\}^d\). Hence
		\[
		\|v\|_{\ell^\infty(G)}
		\lesssim
		\sum_{\bm\alpha\in\{0,1\}^d}\|D^{\bm\alpha}v\|_{\ell_h^2(G)}
		\lesssim h^{k+2},
		\]
		which is the desired estimate.
	\end{proof}

\section{Fully discrete RKCDG superconvergence}\label{sec:RKCDG}

The fully discrete CDG method computes numerical solutions at discrete time levels \( \{t_n\} \) with time step \( \Delta t = t_{n+1} - t_n \). Let \(\langle \cdot,\cdot \rangle\) denote the \(L^2\) inner product on \(\Omega\). Coupled with an \(r\)-th order, \(s\)-stage explicit Runge--Kutta method, the fully discrete CDG scheme is 
\begin{subequations}\label{eq:RKCDG}
	\begin{align}
		\bm{u}_h^{n,0} &= \bm{u}_h^{n}, \\
		\label{def:stagesol}
		\langle \bm{u}_h^{n,l+1}, \bm{\phi} \rangle
		&=
		\sum_{\alpha=0}^{l}
		\left(
		c_{l\alpha} \langle \bm{u}_h^{n,\alpha}, \bm{\phi} \rangle
		+
		\Delta t\, d_{l\alpha} H(\bm{u}_h^{n,\alpha}, \bm{\phi})
		\right),
		\qquad l=0,\dots,s-1.
	\end{align}
\end{subequations}
Here \(\bm{u}_h^{n,l}\in V_h\times W_h\) denotes the CDG stage solution at the \(l\)-th RK stage, \(c_{l\alpha}\) and \(d_{l\alpha}\) are the RK coefficients, and \(H(\cdot,\cdot)\) is the global bilinear form obtained by summing all local spatial operators \(\hat{B}_C\) and \(\tilde{B}_D\).

\begin{proposition}[Stability of the explicit RKCDG method, Theorem 3.5 in \cite{peng2025oscillation}]
	\label[proposition]{prop:stability}
	Under the mesh condition \( \Delta t \lesssim h^\kappa \) with \( \kappa \in [1,2] \), the explicit RKCDG method with source terms satisfies
	\begin{equation*}
		\langle \bm{u}_h^{n,l+1}, \bm{\phi} \rangle
		=
		\sum_{\alpha=0}^{l}
		\left[
		c_{l\alpha} \langle \bm{u}_h^{n,\alpha}, \bm{\phi} \rangle
		+
		\Delta t\, d_{l\alpha}
		\left(
		H(\bm{u}_h^{n,\alpha}, \bm{\phi})
		+
		\langle \bm{g}^{n,\alpha}, \bm{\phi} \rangle
		\right)
		\right],
		\qquad l = 0, \dots, s-1.
	\end{equation*}
	Furthermore, for any integer \(0 \leq \alpha \leq [T/\Delta t]\),
	\begin{equation}\label{eq:stability_estimate}
		\| \bm{u}_h^{n+\alpha} \|^2
		\leq
		(1 + M \Delta t) \| \bm{u}_h^n \|^2
		+
		M\Delta t\sum_{m =0}^{\alpha s-1} \| \bm{g}^{n,m} \|^2.
	\end{equation}
\end{proposition}
The fully discrete analysis is formulated through corrected errors at the RK stages. The stagewise error equation contains the same spatial residual structures as in the semidiscrete analysis, together with the RK local truncation terms. 
\begin{theorem}[Fully discrete corrected-error superconvergence]
	\label[theorem]{thm:superconvergence_RKCDG}
Let $d\in\{1,2\}$, and consider the $\mathbb{Q}^k$ CDG discretization of \eqref{equ:conservation_law} on a uniform Cartesian overlapping mesh in spatial dimension $d$, coupled with an $r$-th order, $s$-stage explicit RK method. Suppose
	$
	u \in C^{r+1}\bigl([0,T]; H^{\sigma_k+1}(\Omega)\bigr)$ with 
	$	\sigma_k := \min\{2k+1,\,k+3\}.
	$ 
	Let $\hat{\bm{\zeta}}^{n,l}$ denote the $\min\{k,2\}$-times corrected stage error. Under the time-step condition $\Delta t \lesssim h^\kappa$ with $\kappa \in [1,2]$, if
	\(
	\|\hat{\bm{\zeta}}^{0,0}\| \lesssim h^{\sigma_k},
	\)
	then
	\[
	\max_{0\le n\le \lceil T/\Delta t\rceil} \|\hat{\bm{\zeta}}^{n,0}\|
	\lesssim h^{\sigma_k} + (\Delta t)^r.
	\]
\end{theorem}
In this statement, \(\hat{\bm\zeta}\) denotes the 1D corrected error when \(d=1\) and the 2D directionally corrected error when \(d=2\).
As in \cref{coro:supercon1D,coro:supercon2D}, the cell-average estimates follow directly from the fully discrete corrected-error estimate. 

\begin{corollary}[Fully discrete cell-average superconvergence]
	\label[corollary]{coro:superconvergence_RKCDG}
	Under the assumptions of \cref{thm:superconvergence_RKCDG}, if the scheme is initialized with the corrected projection, then
	\[
	\max_{0 \le n \le \lceil T/\Delta t\rceil} e_{\mathrm{avg}}^{\,n}
	\lesssim h^{\sigma_k}+(\Delta t)^r.
	\]
\end{corollary}
 The corrected-initialization corollary also gives a reconstruction-based postprocessing result. Let \(\Pi_h^0\) be the \(L^2\)-projection onto primal-mesh piecewise constants. In 1D, define the local reconstruction operator \(S_h\) by primitive interpolation: for \(v\in \Pi_h^0L^2(\Omega)\), let \(\Pi_h\) be the polynomial of degree \(\sigma_k\) that interpolates the discrete primitive
\(
V(x_{i+1/2})=\sum_{j\le i} h\,(\Pi_h^0v)|_{I_j}
\)
at \(\sigma_k+1\) neighboring cell boundaries, and set \(S_hv:=\Pi_h'\). In 2D, \(S_h\) is defined by tensor products of the corresponding 1D reconstructions. Then \(S_h\) is a local bounded linear operator, preserves cell averages, and reproduces polynomials of degree at most \(\sigma_k-1\).

	\subsection{Proof of the fully discrete superconvergence results}
	This subsection proves the fully discrete 1D and 2D RKCDG result including the reference stage solutions, stagewise corrected errors, and stagewise residual estimates. The LSZ-projection-based $\ell^\infty$ pointwise superconvergence theorem is proved with two key ingredients: $\ell_h^2$ difference quotient estimate and a discrete Sobolev inequality. In the fully discrete setting, it is important to show the same asymptotic weak-cancellation and HOCA mechanisms remain valid for stage solutions.

	\subsubsection*{Reference functions:}
	
	Define the reference function \(\bm{U}^{n,l}(\bm{x}) := [u^{n,l}(\bm{x}), u^{n,l}(\bm{x})]^T\), where \(u^{n,l}(\bm{x})\) satisfies
	\begin{align}
		u^{n,0} &= u(\bm{x}, t_n), \nonumber \\
		u^{n,l+1}
		&=
		\sum_{\alpha=0}^{l}
		\left(
		c_{l\alpha} u^{n,\alpha}
		-\Delta t\, d_{l\alpha} \bm{\beta} \cdot \nabla u^{n,\alpha}
		\right)
		+
		\Delta t \rho^{n,l},
		\qquad l = 0, \dots, s-1,
		\label{eq:reference_u}
	\end{align}
	with \(\rho^{n,l} = 0\) for \(l < s-1\) and \(\rho^{n,s-1}\) chosen so that \(u^{n,s} = u(\bm{x}, t_{n+1})\). By the RK order conditions,
	$
	\|\rho^{n,s-1}\| \lesssim (\Delta t)^r.
	$ 
	Let \(\bm{\rho}^{n,l} := [\rho^{n,l}, \rho^{n,l}]^T\). Since \(u^{n,l}\) is smooth,
	\[
	H(\bm{U}^{n,l}, \bm{\phi}) = - \langle \bm{\nabla}_{\bm{\beta}} \bm{U}^{n,l}, \bm{\phi} \rangle,
	\qquad
	\bm{\nabla}_{\bm{\beta}} \bm{U}^{n,l}
	=
	[\bm{\beta}\cdot\nabla u^{n,l}, \bm{\beta}\cdot\nabla u^{n,l}]^T,
	\]
	and therefore
	\begin{equation}\label{def:refsol}
		\langle \bm{U}^{n,l+1}, \bm{\phi} \rangle
		=
		\sum_{\alpha=0}^{l}
		\left(
		c_{l\alpha} \langle \bm{U}^{n,\alpha}, \bm{\phi} \rangle
		+
		\Delta t\, d_{l\alpha} H(\bm{U}^{n,\alpha}, \bm{\phi})
		\right)
		+
		\Delta t \langle \bm{\rho}^{n,l}, \bm{\phi} \rangle.
	\end{equation}
	
	For each RK stage, define
	\begin{equation}\label{equ:stageerror}
		\hat{\bm{\zeta}}^{n,l}
		=
		\bm{\zeta}^{n,l} + \sum_{\alpha=1}^{\min\{k,2\}} \bm{R}^\alpha u^{n,l},
		\qquad
		\hat{\bm{\eta}}^{n,l}
		=
		\bm{\eta}^{n,l} + \sum_{\alpha=1}^{\min\{k,2\}} \bm{R}^\alpha u^{n,l},
	\end{equation}
	where
	$
	\bm{\zeta}^{n,l} = \bm{u}_h^{n,l}-\bm{P}u^{n,l}$ and 
	$
	\bm{\eta}^{n,l} = \bm{U}^{n,l}-\bm{P}u^{n,l}.
	$

	\subsection{Fully discrete correction functions}
	
	At each RK stage, use the correction functions from the main text with \(u\) replaced by the reference stage solution \(u^{n,l}\). Define the global correction operator in 1D by
	\[
	\tilde{H}(\bm{R}^{\alpha} w, \bm{\phi})
	:=
	\sum_i
	\left(
	\widehat{P}_i(R_C^{\alpha} w, \phi_1)
	+
	\widetilde{P}_{i+1/2}(R_D^{\alpha} w, \phi_2)
	\right),
	\]
	and in 2D by
	$
	\tilde{H}(\bm{R}^{\alpha} w, \bm{\phi})
	:=
	\widetilde{H}_x(\bm{X}^{\alpha}w, \bm{\phi})
	+
	\widetilde{H}_y(\bm{\Gamma}^{\alpha}w, \bm{\phi})$ with $
	\bm{R}^{\alpha}w=\bm{X}^{\alpha}w+\bm{\Gamma}^{\alpha}w,
	$
	where
	\begin{align*}
		\widetilde{H}_x(\bm{w}, \bm{\phi})
		&	=
		\sum_{i,j}\int_{J_j}
		\left(
		\widehat{P}_{i}^x(w_C, \phi_1)
		+
		\widetilde{P}_{i+1/2}^x(w_D, \phi_2)
		\right)\mathrm{d}y, \\
		\widetilde{H}_y(\bm{w}, \bm{\phi})&	=
		\sum_{i,j}\int_{I_i}
		\left(
		\widehat{P}_{j}^y(w_C, \phi_1)
		+
		\widetilde{P}_{j+1/2}^y(w_D, \phi_2)
		\right)\mathrm{d}x.
	\end{align*}
	
	\subsection{Proof of \Cref{thm:superconvergence_RKCDG}: stagewise corrected-error estimate}
	
	Subtracting \eqref{def:refsol} from \eqref{def:stagesol} and including the correction functions, we obtain
	\begin{equation}\label{def:errorequ1}
		\langle \hat{\bm{\zeta}}^{n,l+1}, \bm{\phi} \rangle
		=
		\sum_{\alpha=0}^{l}
		\big(
		c_{l\alpha} \langle \hat{\bm{\zeta}}^{n,\alpha}, \bm{\phi} \rangle
		+
		\Delta t\, d_{l\alpha} H(\hat{\bm{\zeta}}^{n,\alpha}, \bm{\phi})
		\big)
		+
		\Delta t \mathcal{G}^{n,l}(\bm{\phi}),
	\end{equation}
	where
	$
	\mathcal{G}^{n,l}(\bm{\phi})
	=
	\langle \hat{\bm{\eta}}_c^{n,l}, \bm{\phi} \rangle
	-
	H(\hat{\bm{\eta}}_d^{n,l}, \bm{\phi})
	-
	\langle \bm{\rho}^{n,l}, \bm{\phi} \rangle.
	$
	
	For any stage function \(\bm{w}^{n,l}\), following \cite{xu_meng_shu_zhang_2020}, define
	$$
	\bm{w}_c^{n,l}
	=
	\frac{1}{\Delta t}
	\left(
	\bm{w}^{n,l+1} - \sum_{\alpha=0}^l c_{l\alpha} \bm{w}^{n,\alpha}
	\right) \quad \mbox{and} \quad 
	\bm{w}_d^{n,l}
	=
	\sum_{\alpha=0}^l d_{l\alpha} \bm{w}^{n,\alpha}.
	$$ 
	This notation is unrelated to the upper-case subscripts \(C,D\) used for primal and dual meshes.
	
	To apply \Cref{prop:stability}, we rewrite \eqref{def:errorequ1} as
	\begin{equation}\label{def:errorequ2}
		\langle \hat{\bm{\zeta}}^{n,l+1}, \bm{\phi} \rangle
		=
		\sum_{\alpha=0}^l
		\big[
		c_{l\alpha} \langle \hat{\bm{\zeta}}^{n,\alpha}, \bm{\phi} \rangle
		+
		\Delta t\, d_{l\alpha}
		\big(
		H(\hat{\bm{\zeta}}^{n,\alpha}, \bm{\phi})
		+
		\bm{G}^{n,\alpha}(\bm{\phi})
		\big)
		\big],
	\end{equation}
	where the source term \(\bm{G}^{n,l}(\bm{\phi})\) is defined recursively by 
	$$
	d_{ll}\,\bm{G}^{n,l}(\bm{\phi})
	=
	\mathcal{G}^{n,l}(\bm{\phi})
	-
	\sum_{\alpha=0}^{l-1} d_{l\alpha}\,\bm{G}^{n,\alpha}(\bm{\phi}),
	0 \le l \le s-1.
	$$ 
	Hence
	\[
	\|\bm{G}^{n,l}\|
	\lesssim
	\sum_{0 \le \alpha\le l} \|\mathcal{G}^{n,\alpha}\|,
	\qquad
	\|\mathcal{G}^{n,\alpha}\|
	:=
	\sup_{\bm{\phi} \in V_h \times W_h,\ \bm{\phi} \ne 0}
	\frac{|\mathcal{G}^{n,\alpha}(\bm{\phi})|}{\|\bm{\phi}\|}.
	\]
	
	Since \(\bm{P}\) and \(\bm{R}\) are linear,
	\begin{equation*}
		\resizebox{0.99\hsize}{!}{$
			\mathcal{G}^{n,l}(\bm{\phi})
			=
			\left\langle
			\left(I - \bm{P} + \sum_{\alpha=1}^{\min\{k,2\}} \bm{R}^\alpha\right) u_c^{n,l},
			\bm{\phi}
			\right\rangle
			-
			H\left(
			\left(I - \bm{P} + \sum_{\alpha=1}^{\min\{k,2\}} \bm{R}^\alpha\right) u_d^{n,l},
			\bm{\phi}
			\right)
			-
			\langle \bm{\rho}^{n,l}, \bm{\phi} \rangle.$}
	\end{equation*}
	We split
	$
	\mathcal{G}^{n,l}(\bm{\phi}) = \mathcal{I}_1^{n,l}(\bm{\phi}) + \mathcal{I}_2^{n,l}(\bm{\phi}),
	$ 
	where
	\begin{subequations}
		\begin{align}
			\label{equ:fd1_a}
			\mathcal{I}_1^{n,l}(\bm{\phi})
			&=
			\left\langle -\bm{R}^{\min\{k,2\}} u_c^{n,l}, \bm{\phi} \right\rangle, \\
			\label{equ:fd1_b}
			\mathcal{I}_2^{n,l}(\bm{\phi})
			&=
			\sum_{\alpha=1}^{\min\{k,2\}}
			\left[
			-\left\langle \bm{R}^{\alpha-1} u_c^{n,l}, \bm{\phi} \right\rangle
			+
			(H - \tilde{H})\big( \bm{R}^{\alpha-1} u_d^{n,l}, \bm{\phi} \big)
			+
			\tilde{H}(\bm{R}^{\alpha} u_d^{n,l}, \bm{\phi})
			\right].
		\end{align}
	\end{subequations}
	
	\subsubsection*{Estimate of \(\boldsymbol{|\mathcal{I}_1^{n,l}(\bm{\phi})|}\):}
	
	Applying \Cref{prop:correctestimate,prop:correctestimate2D} to \(u_c^{n,l}\), we obtain
	\[
	|\mathcal{I}_1^{n,l}(\bm{\phi})|
	\lesssim
	h^{\sigma_k} \|u_c^{n,l}\|_{H^{\sigma_k}} \|\bm{\phi}\|
	\lesssim
	h^{\sigma_k} \|\bm{\phi}\|.
	\]
	
	\subsubsection*{Estimate of \(\boldsymbol{|\mathcal{I}_2^{n,l}(\bm{\phi})|}\):}
	
	We present the 2D case; the 1D case is recovered by keeping only the \(x\)-direction terms and replacing the inner product and spatial operator by their 1D counterparts from the semidiscrete analysis. 
	\begin{align}
		\mathcal{I}_2^{n,l}(\bm{\phi})
		&=
		\mathcal{I}_x^{n,l}(\bm{\phi})
		+
		\mathcal{I}_y^{n,l}(\bm{\phi})
		+
		\mathcal{I}_{xy}^{n,l}(\bm{\phi}), \nonumber \\
		\mathcal{I}_x^{n,l}(\bm{\phi})
		&=
		\sum_{\alpha=1}^{\min\{k,2\}}
		\Big[
		\langle \bm{X}^{\alpha-1} (\bm{\beta}\cdot\nabla u_d^{n,l}), \bm{\phi} \rangle
		+
		(H - \tilde{H}_x)\big( \bm{X}^{\alpha-1} u_d^{n,l}, \bm{\phi} \big)
		+
		\tilde{H}_x\big(\bm{X}^{\alpha} u_d^{n,l}, \bm{\phi}\big)
		\Big], \label{eq:Ix-decomp} \\
		\mathcal{I}_y^{n,l}(\bm{\phi})
		&=
		\sum_{\alpha=1}^{\min\{k,2\}}
		\Big[
		\langle \bm{\Gamma}^{\alpha-1} (\bm{\beta}\cdot\nabla u_d^{n,l}), \bm{\phi} \rangle
		+
		(H - \tilde{H}_y)\big( \bm{\Gamma}^{\alpha-1} u_d^{n,l}, \bm{\phi} \big)
		+
		\tilde{H}_y\big(\bm{\Gamma}^{\alpha} u_d^{n,l}, \bm{\phi}\big)
		\Big], \notag \\
		\mathcal{I}_{xy}^{n,l}(\bm{\phi})
		&=
		\sum_{\alpha=1}^{\min\{k,2\}}
		\Big[
		-\langle \bm{R}^{\alpha}\big( u_c^{n,l} + \bm{\beta}\cdot\nabla u_d^{n,l} \big), \bm{\phi} \rangle
		\Big]
		+
		H\big( \tilde{\bm{\eta}}_d^{n,l}, \bm{\phi} \big). \notag
	\end{align}
	
	We first estimate the \(x\)-direction terms. From the semidiscrete analysis,
	\begin{equation}\label{eq:x-block-bound}
		\resizebox{0.99\hsize}{!}{$
			\left|
			\sum_{\alpha = 1}^{\min\{k,2\}}
			\Big[
			\langle \bm{X}^{\alpha-1} (\bm{\beta}\cdot\nabla u_d^{n,l}), \bm{\phi} \rangle
			+
			(H - \tilde{H}_x)\big( \bm{X}^{\alpha-1}u_d^{n,l}, \bm{\phi} \big)
			+
			\tilde{H}_x\big(\bm{X}^{\alpha} u_d^{n,l}, \bm{\phi}\big)
			\Big]
			\right|
			\lesssim
			h^{\sigma_k} \|\bm{\phi}\|.$}
	\end{equation}
	Therefore, by \eqref{eq:Ix-decomp},
	\[
	\begin{aligned}
		\big|\mathcal{I}_x^{n,l}(\bm{\phi})\big|
		&\lesssim
		h^{\sigma_k} \|\bm{\phi}\|
		+
		\left|
		\sum_{\alpha=0}^{\min\{k,2\}-1}
		\Big\langle
		\bm{X}^{\alpha}\big( u_c^{n,l} + \bm{\beta}\cdot\nabla u_d^{n,l}\big),
		\bm{\phi}
		\Big\rangle
		\right| \lesssim
		h^{\sigma_k} \|\bm{\phi}\|
		+
		\|\bm{\rho}^{n,l}\| \|\bm{\phi}\|,
	\end{aligned}
	\]
	where the last step uses \eqref{eq:reference_u}. Hence
	$
	\big|\mathcal{I}_x^{n,l}(\bm{\phi})\big|
	\lesssim
	\bigl((\Delta t)^r + h^{\sigma_k}\bigr)\|\bm{\phi}\|.
	$ 
	The estimate for \(\mathcal{I}_y^{n,l}\) is identical.
	
	For the mixed term,
	$$
	\big|\mathcal{I}_{xy}^{n,l}(\bm{\phi})\big|
	\lesssim
	\|\bm{\rho}^{n,l}\| \|\bm{\phi}\|
	+
	h^{-1} \big\| (I - \bm{P}^x)\big[(I - \bm{P}^y) u_d^{n,l}\big] \big\| \|\bm{\phi}\|
	\lesssim
	\bigl(h^{\sigma_k} + (\Delta t)^r\bigr)\|\bm{\phi}\|.
	$$
	
	Combining the above bounds for \(\mathcal{G}^{n,l}\) with the recursion \eqref{def:errorequ2}, and then applying \Cref{prop:stability}, we obtain
	$
	\|\hat{\bm{\zeta}}^{n,l}\|
	\lesssim
	\|\hat{\bm{\zeta}}^{0}\|
	+
	h^{\sigma_k}
	+
	(\Delta t)^r.
	$
	Choosing the initial value so that
	$
	\|\hat{\bm{\zeta}}^0\| \lesssim h^{\sigma_k},
	$ 
	and applying the discrete Gr\"onwall inequality completes the proof of \Cref{thm:superconvergence_RKCDG}. The corollary follows from the same corrected-projection initialization used in the semidiscrete case.

	\subsection{Reconstruction-based postprocessing}
\begin{lemma}[Stable reconstruction-based postprocessing]
	\label[lemma]{lem:postprocessing}
	Under the hypotheses of \Cref{coro:superconvergence_RKCDG}, the postprocessed fully discrete RKCDG approximation via reconstruction satisfies
	\[
	\max_{0 \le n \le \lceil T/\Delta t\rceil}
	\bigl\|u(\cdot,t^n)-S_h\Pi_h^0 u_h^{C,n}\bigr\|
	\lesssim
	h^{\sigma_k}+(\Delta t)^r.
	\]
\end{lemma}

\begin{proof}
	Let $u^n = u(\cdot, t^n)$. By the triangle inequality, the polynomial preservation of $S_h$, and its $L^2$-stability, we have
	\begin{align*}
		\|u^n - S_h\Pi_h^0 u_h^{C,n}\| &\le \|u^n - S_h\Pi_h^0 u^n\| + \|S_h\Pi_h^0 (u^n - u_h^{C,n})\| \\
		&\lesssim h^{\sigma_k}\|u^n\|_{H^{\sigma_k}(\Omega)} + \|\Pi_h^0 (u^n - u_h^{C,n})\| 
		\lesssim h^{\sigma_k} + e_{\mathrm{avg}}^{\,n}.
	\end{align*}
	Substituting $e_{\mathrm{avg}}^{\,n} \lesssim h^{\sigma_k} + (\Delta t)^r$ from \Cref{coro:superconvergence_RKCDG} completes the proof.
\end{proof}

As for the LSZ projection initialization, the initial corrected error is only of order \(h^{k+2}\). The fully discrete error equation and the corresponding discrete Gr\"onwall argument therefore yield the \(k+2\) order $\ell_h^2$ pointwise estimate, together with the same estimate for the difference quotients, up to the temporal truncation error
\[
\max_n
\sum_{\bm\alpha\in\{0,1\}^d}
\|D^{\bm\alpha}(u(\cdot,t^n)-u_h^{C,n})\|_{\ell_h^2(G)}
\lesssim h^{k+2}+(\Delta t)^r .
\]
Applying \Cref{lem:discrete_sobolev_roots} gives the desired fully discrete \(\ell^\infty\) pointwise estimate.
\begin{lemma}[Fully discrete LSZ-projection-based $\ell^\infty$ point estimate]\label[lemma]{coro:base_LSZ_RK}
Under the assumptions of \Cref{thm:base_LSZ_maximum}, with the additional time regularity required in \Cref{thm:superconvergence_RKCDG}, and under the RKCDG stability condition, suppose the fully discrete RKCDG method is initialized by the uncorrected projection
\(
\bm u_h^{0,0}=\bm{P}u(\cdot,0).
\)
Then
\[
\max_{0\le n\le \lceil T/\Delta t\rceil}
\max_{\xi\in G}
|u(\xi,t^n)-u_{h}^{C,n}(\xi)|
\lesssim h^{k+2}+(\Delta t)^r.
\]
\end{lemma}

	\begin{proof}
		The proof of \Cref{thm:superconvergence_RKCDG} actually implies 
	$$
	\max_{0\le n\le N}\|\hat{\bm\zeta}^{n,0}\|
	\lesssim
	\|\hat{\bm\zeta}^{0,0}\|+h^{\sigma_k}+(\Delta t)^r.
	$$
	Under the uncorrected LSZ projection initialization,
	$
	\|\hat{\bm\zeta}^{0,0}\|\lesssim h^{k+2},
	$ 
	and hence 
	$$
	\max_{0\le n\le N}\|\hat{\bm\zeta}^{n,0}\|
	\lesssim h^{k+2}+(\Delta t)^r.
	$$
	The same projection-root estimate, correction decomposition, and sampling
	argument used in the proof of \Cref{lem:l2point}, applied also to
	\(D^{\bm\alpha}u\), yield
	\[
	\max_n\sum_{\bm\alpha\in\{0,1\}^d}
	\|D^{\bm\alpha}(u(\cdot,t^n)-u_h^{C,n})\|_{\ell_h^2(G)}
	\lesssim h^{k+2}+(\Delta t)^r.
	\]
	Finally, \Cref{lem:discrete_sobolev_roots} gives \Cref{coro:base_LSZ_RK}.
\end{proof}

\section{Numerical examples}\label{sec:Numtest}

	This section verifies the corrected-initialization consequences of the theory, namely the cell-average superconvergence estimates and the reconstruction-based postprocessing estimate. Using the LSZ projection for initialization, the \(\ell^\infty\)-norm pointwise theorem in \Cref{thm:base_LSZ_maximum} explains the LSZ pointwise superconvergent phenomenon under the uncorrected projection documented in \cite{liu2018optimal}; the experiments below focus on the stronger cell-average and postprocessing consequences that require corrected initialization. The tables are therefore intended to test the new corrected-initialization consequences rather than to repeat the LSZ pointwise tests. We present 1D and 2D experiments for the \(\mathbb Q^k\) CDG method applied to the linear advection equation \eqref{equ:conservation_law}. The initial conditions are smooth and satisfy the regularity assumptions of the analysis. To match the corrected-error variables used in the theory, the initial data are taken as \(\bm P^\ast u_0-\bm R^1u_0\) in 1D and as \(\bm Q^\ast u_0-\bm X^1u_0-\bm\Gamma^1u_0\) in 2D. Time integration uses the seventh-order linear strong-stability-preserving Runge--Kutta scheme \cite{gottlieb2005high}.

The seventh-order SSP Runge--Kutta scheme is used only to reduce the temporal
error in the spatial convergence tests. Although CDG schemes can admit larger
CFL numbers than standard DG methods \cite{liu2008l2}, the explicit CFL
constants reported in \cite{liu2008l2} are given only up to fourth-order
Runge--Kutta methods. Hence we use the conservative choice 
$
C_{\mathrm{CFL}}=\frac{1}{2k+1}
$ 
for all tests with the seventh-order SSP RK scheme. In 1D, we set \(\Delta t=C_{\mathrm{CFL}}h\). In 2D, we take
\(
\Delta t=\frac{C_{\mathrm{CFL}}}{h_x^{-1}+h_y^{-1}}
=C_{\mathrm{CFL}}\frac{h_xh_y}{h_x+h_y}.
\)
All computations are implemented in \texttt{C/C++} in double precision.

\subsection{1D example}

We consider
$
u_0(x) = \sin(2\pi x)
$ 
on the domain \(\Omega=[0,1]\) with periodic boundary conditions. Uniform meshes with \(N_x\) cells are used.
Table~\ref{tab:1D} reports the cell-average error $e_{\mathrm{avg}}$ and the postprocessing error $\|u(\cdot,T)-S_h\Pi_h^0 u_h^{C,n}\|$ at time $t=1.1$, using $\mathbb{Q}^k$ CDG elements on a uniform mesh with $N_x$ cells. Both errors converge with order about $3$, $5$, and $6$ for $k=1,2,3$, respectively, confirming the theoretical rate $\min\{2k+1, k+3\}$. In particular, for $k \ge 2$, we clearly observe the improved $(k+3)$-th order convergence when the initial data are corrected once.

\begin{table}[!htb]
	\centering
	\belowrulesep=0pt
	\aboverulesep=0pt
	\caption{Errors and convergence rates for the 1D $\mathbb{Q}^k$-based CDG method with $N_x$ cells.}
	\label{tab:1D}
	\setlength{\tabcolsep}{3mm}{
		\begin{tabular}{c|cccccc}
			\toprule[1.2pt]
			\multirow{2}{*}{} &
			\multirow{2}{*}{$N_x$} &
			\multicolumn{2}{c}{$e_{\text{avg}}$}&	\multicolumn{2}{c}{$\|u(\cdot,T)-S_h\Pi_h^0 u_h^{C,n}\|$} \\
			\cmidrule(r){3-4} \cmidrule(l){5-6}
			& & value & order & value & order \\
			\midrule[1.5pt]
			\multirow{5}{*}{$k=1$}
			&60& 6.95e-05 & --& 8.42e-05 & --\\ 
			&70& 4.38e-05 & 2.99& 5.30e-05 & 3.01\\ 
			&80& 2.93e-05 & 2.99& 3.54e-05 & 3.01\\ 
			&90 & 2.06e-05 & 3.00 & 2.49e-05 & 3.01\\ 
			&100 & 1.50e-05 & 3.00& 1.81e-05 & 3.01\\ 
			\midrule
			\multirow{5}{*}{$k=2$}
			&60 & 3.94e-09 & --& 9.74e-08 & --\\
			&70 & 1.82e-09 & 5.02& 4.51e-08 & 5.00\\
			&80 & 9.32e-10 & 5.02& 2.31e-08 & 5.00\\
			&90 & 5.16e-10 & 5.01& 1.28e-08 & 5.00\\
			&100 & 3.04e-10 & 5.01& 7.57e-09 & 5.00\\
			\midrule
			\multirow{5}{*}{$k=3$}
			&20 & 4.64e-09 & --& 3.68e-06 &  --\\ 
			&30 & 3.92e-10 & 6.10 & 3.27e-07 & 5.97\\ 
			&40 & 6.84e-11 & 6.06& 5.83e-08 & 5.99\\ 
			&50 & 1.78e-11 & 6.05 & 1.53e-08 & 5.99\\ 
			&60 & 5.90e-12 & 6.04& 5.14e-09 & 5.99\\
			\bottomrule[1.5pt]
		\end{tabular}
	}
\end{table}

\subsection{2D example}

In 2D, we consider a stationary problem of \eqref{equ:conservation_law} on $\Omega = [0,1]^2$ with periodic boundary conditions and initial data
$
u_0(x,y) = \sin\big( 2\pi (x-y) \big).
$ 
Table~\ref{tab:2D} reports the cell-average and the postprocessing superconvergent errors at $t = 1.1$ using the CDG method on $N_x \times N_y$ uniform cells. These results confirm the 2D theory presented in \cref{coro:supercon2D} and \Cref{lem:postprocessing}.

\begin{table}[!htb]
	\centering
	\belowrulesep=0pt
	\aboverulesep=0pt
	\caption{Errors and convergence rates for 2D $\mathbb{Q}^k$-based CDG with $N_x \times N_y$ cells.}
	\label{tab:2D}
	\setlength{\tabcolsep}{3mm}{
			\begin{tabular}{c|ccccc}
				\toprule[1.2pt]
				\multirow{2}{*}{} &
				\multirow{2}{*}{$N_x \times N_y$} &
				\multicolumn{2}{c}{$e_{\text{avg}}$}& \multicolumn{2}{c}{$\|u(\cdot,T)-S_h\Pi_h^0 u_h^{C,n}\|$} \\
				\cmidrule(r){3-4} \cmidrule(l){5-6}
				& & value & order & value & order\\
			\midrule[1.5pt]
			\multirow{6}{*}{$k=1$}
				& 40 $\times$ 40  & 4.58e-04 & --& 5.14e-04 & --\\ 
			& 50 $\times$ 50 & 2.36e-04 & 2.97& 2.63e-04 & 3.00\\  
			& 60 $\times$ 60  & 1.37e-04 & 2.98& 1.53e-04 & 3.00\\ 
			& 70 $\times$ 70 & 8.66e-05 & 2.98& 9.61e-05 & 3.00\\ 
			& 80 $\times$ 80  & 5.82e-05 & 2.98& 6.44e-05 & 3.00\\  
			& 90 $\times$ 90  & 4.09e-05 & 2.99& 4.52e-05 & 3.00\\ 
			\midrule
			\multirow{6}{*}{$k=2$}
			& 40 $\times$ 40 & 5.27e-08 & --& 1.05e-06 & --\\ 
			& 50 $\times$ 50  & 1.70e-08 & 5.08& 3.45e-07 & 5.00\\ 
			& 60 $\times$ 60 & 6.70e-09 & 5.10& 1.38e-07 & 5.00\\ 
			& 70 $\times$ 70  & 3.04e-09 & 5.13& 6.41e-08 & 5.00\\ 
			& 80 $\times$ 80  & 1.53e-09 & 5.15& 3.29e-08 & 5.00\\ 
			& 90 $\times$ 90  & 8.30e-10 & 5.18& 1.82e-08 & 5.00\\ 
			\midrule
			\multirow{5}{*}{$k=3$}
			& 20 $\times$ 20  & 3.82e-09 & --& 5.18e-06 &--\\ 
			& 30 $\times$ 30  & 3.46e-10 & 5.92& 4.61e-07 & 5.97\\ 
			& 40 $\times$ 40  & 6.29e-11 & 5.93& 8.24e-08 & 5.98\\ 
			& 50 $\times$ 50  & 1.67e-11 & 5.94& 2.17e-08 & 5.99\\ 
			& 60 $\times$ 60  & 5.68e-12 & 5.92& 7.26e-09 & 5.99\\ 
			\bottomrule[1.5pt]
		\end{tabular}
	}
\end{table}

\section{Conclusions}\label{sec:conclusions}

In this work, we developed the first superconvergence analysis for central discontinuous Galerkin (CDG) methods applied to linear hyperbolic equations on overlapping meshes. The $(k+2)$nd-order pointwise behavior of the $\mathbb{Q}^{k}$ CDG method had been observed numerically in  \cite{liu2018optimal}, but  a theoretical proof had remained elusive, primarily because the CDG spatial operator on overlapping meshes lacks the Galerkin orthogonality properties inherent to standard upwind DG schemes.

We have overcome this obstacle by introducing a projection-correction framework and directional correction functions. Central to our analysis is the discovery of a hidden residual cancellation mechanism governing CDG superconvergence: an asymptotic weak constant-test cancellation in one dimension, and an innovative high-order cancellation-by-aggregation (HOCA) mechanism in two dimensions. 
This discovery enabled a rigorous proof of the conjectured $\mathcal{O}(h^{k+2})$ pointwise superconvergence in the discrete $\ell^\infty$-norm under the uncorrected Liu--Shu--Zhang (LSZ) initialization, using an anchored discrete Sobolev inequality. Beyond validating these expected pointwise rates, the proposed framework revealed a previously undiscovered, stronger superconvergence phenomenon. Conditioned on a corrected initialization, we established a higher-order cell-average estimate of order $\mathcal{O}(h^{\sigma_k})$ with $\sigma_k=\min\{2k+1,k+3\}$, a finding that surpasses previous expectations.

The theory was extended to fully discrete explicit RKCDG schemes. By constructing stagewise discrete correction functions, our analysis showed that the spatial superconvergent orders are preserved up to temporal truncation errors under suitable time-step conditions. Combined with the highly accurate cell averages, this extension results in a stable reconstruction-based postprocessing estimate of the same order, $\mathcal{O}(h^{\sigma_k})$. Numerical experiments in one and two spatial dimensions confirmed the predicted cell-average and postprocessed convergence rates.

The model considered here is the canonical one for classical pointwise superconvergence: smooth evolution, explicit superconvergent point sets, and Cartesian (primal--dual overlapping) meshes. It is also the setting in which the essential CDG difficulty is visible in its cleanest form. In one dimension the proof is driven by a weak constant-test cancellation. In two dimensions this cancellation is not inherited by each directional residual separately, because the tensor-product projection has no single local variational characterization and the primal--dual overlap couples transverse and mixed terms. The HOCA argument shows that the sharp order is recovered only after the projection and same-direction correction residuals are aggregated. This feature explains why the multidimensional theorem is not a routine tensorization of the one-dimensional analysis.

The analysis therefore provides structural targets for extensions. For smooth nonlinear conservation laws, this amounts to combining the present projection--correction cancellations with nonlinear flux-expansion and commutator estimates; for variable coefficients, boundary effects, and non-Cartesian meshes, it means identifying replacements for the special point/projection structure and for the HOCA aggregation identities. These directions are natural continuations of the present work, and the theorems proved here supply a benchmark against which such extensions can be measured.

		\bibliography{refs}
		\bibliographystyle{siamplain}

\end{document}